\documentclass[
paper=letter,%
numbers=noendperiod,%
captions=nooneline,%
abstracton,%
twoside%
]{scrartcl}

\usepackage[T1]{fontenc}%
\usepackage{lmodern}%
\usepackage[american]{babel}%
\usepackage{microtype}%

\usepackage[
hyperref,%
table%
]{xcolor}%

\usepackage[includeheadfoot, top=0.7in, bottom=0.35in, left=0.8in, right=0.8in]{geometry}%
\usepackage{scrlayer-scrpage}%

\usepackage{eso-pic}%
\usepackage{rotating}%
\usepackage{amsmath}%
\usepackage{mathtools}%
\usepackage{amssymb}%
\usepackage{amsthm}%
\usepackage{thmtools}%
\usepackage{etoolbox}%
\usepackage{bm}%
\usepackage{bbm}%
\usepackage{enumitem}%
\usepackage{graphicx}%
\usepackage{grffile}%
\usepackage{tikz}%
\usetikzlibrary{positioning}%
\usepackage{wrapfig}%
\usepackage{tabularx}%
\usepackage{siunitx}%
\usepackage{booktabs}%
\usepackage{multirow}%
\usepackage{vruler}%
\usepackage{fancyvrb}%
\usepackage{listings}%
\usepackage{csquotes}%
\usepackage{comment}%
\usepackage[
backend=bibtex,%
style=authoryear,%
dashed=false,%
hyperref=true,%
useprefix=true,%
maxcitenames=2,%
maxbibnames=6,%
giveninits=true,%
]{biblatex}%
\usepackage[
hypertexnames=false,%
setpagesize=false,%
pdfborder={0 0 0},%
pdfstartview=Fit,%
bookmarksopen=true,%
bookmarksnumbered=true%
]{hyperref}%
\xdefinecolor{lightgray}{RGB}{247, 247, 247}%
\xdefinecolor{semilightgray}{RGB}{240, 240, 240}%
\xdefinecolor{middlegray}{RGB}{127, 127, 127}%
\xdefinecolor{blue}{RGB}{58, 95, 205}%
\xdefinecolor{deepskyblue}{RGB}{0, 154, 205}%
\xdefinecolor{chocolate}{RGB}{205, 102, 29}%

\newcommand*{\cc}[1]{\cellcolor[gray]{#1}}

\clearscrheadfoot%
\cohead{\normalfont\footnotesize COMPATIBILITY AND ATTAINABILITY OF MATRICES OF CORRELATION-BASED MEASURES OF CONCORDANCE}%
\rohead{\pagemark}%
\cehead{\normalfont MARIUS HOFERT AND TAKAAKI KOIKE}%
\lehead{\pagemark}%
\setkomafont{pagenumber}{\normalfont}%
\setkomafont{captionlabel}{\normalfont\normalcolor\scshape}%
\setkomafont{disposition}{\normalcolor\scshape}%
\setkomafont{subsection}{\normalcolor\bfseries}%
\setkomafont{subsubsection}{\normalcolor\itshape}%
\setlist{%
  align=left,%
  labelsep=*,%
  leftmargin=*,%
  topsep=1mm,%
  itemsep=0mm%
}
\newcommand*{\mysquare}{\rule[0.18em]{0.36em}{0.36em}}
\newcommand*{\mytriangle}{\raisebox{0.12em}{\resizebox{0.48em}{0.48em}{$\blacktriangleright$}}}
\newcommand*{\mybar}{\rule[0.32em]{0.62em}{0.08em}}
\newcommand*{\mydot}{\raisebox{0.14em}{\resizebox{0.44em}{!}{$\bullet$}}}
\setlist[itemize,1]{label={\mysquare\ }}%
\setlist[itemize,2]{label={\mytriangle\ }}%
\setlist[itemize,3]{label={\mybar\ }}%
\setlist[itemize,4]{label={\mydot\ }}%
\setlist[enumerate,1]{label=\arabic*)}%
\setlist[enumerate,2]{label=\arabic{enumi}.\arabic*)}%
\setlist[enumerate,3]{label=\arabic{enumi}.\arabic{enumii}.\arabic*)}%

\makeatletter
\newcommand\myisodate{\number\year-\ifcase\month\or 01\or 02\or 03\or 04\or 05\or 06\or 07\or 08\or 09\or 10\or 11\or 12\fi-\ifcase\day\or 01\or 02\or 03\or 04\or 05\or 06\or 07\or 08\or 09\or 10\or 11\or 12\or 13\or 14\or 15\or 16\or 17\or 18\or 19\or 20\or 21\or 22\or 23\or 24\or 25\or 26\or 27\or 28\or 29\or 30\or 31\fi}%
\makeatother
\newcommand*{\abstractnoindent}{}%
\let\abstractnoindent\abstract
\renewcommand*{\abstract}{\let\quotation\quote\let\endquotation\endquote
  \abstractnoindent}
\deffootnote[1em]{1em}{1em}{\textsuperscript{\thefootnotemark}}%
\lstdefinestyle{input}{
  backgroundcolor=\color{semilightgray},%
  commentstyle=\itshape\color{chocolate},%
  keywordstyle=\color{blue},%
  stringstyle=\color{blue},%
  numbers=left,%
  numbersep=4.8pt,%
  numberstyle=\color{darkgray!80}\tiny%
}
\lstdefinestyle{output}{
  backgroundcolor=\color{lightgray}%
}

\lstdefinestyle{Lstyle}{
  language=[LaTeX]TeX,%
  texcs={},%
  otherkeywords={}%
}

\lstnewenvironment{Linput}[1][]{%
  \lstset{style=input, style=Lstyle}
  #1%
}{\vspace{-0.25\baselineskip}}%

\lstnewenvironment{Loutput}[1][]{%
  \lstset{style=output, style=Lstyle}
  #1%
}{\vspace{-0.25\baselineskip}}%

\lstdefinestyle{Rstyle}{
  language=R,%
  keywords={if, else, repeat, while, function, for, in, next, break},%
  otherkeywords={}%
}

\lstnewenvironment{Rinput}[1][]{%
  \lstset{style=input, style=Rstyle}
  #1%
}{\vspace{-0.25\baselineskip}}%

\lstnewenvironment{Routput}[1][]{%
  \lstset{style=output, style=Rstyle}
  #1%
}{\vspace{-0.25\baselineskip}}%

\DeclareNameAlias{sortname}{last-first}%
\DefineBibliographyExtras{american}{\DeclareQuotePunctuation{}}%
\renewbibmacro*{volume+number+eid}{%
  \setunit*{\addcomma\space}%
  \printfield{volume}%
  \printfield{number}}
\AtEveryBibitem{}%
\AtEveryBibitem{}%
\DeclareFieldFormat[article]{volume}{\textbf{#1}}%
\DeclareFieldFormat*{number}{(#1)}
\DeclareFieldFormat*{title}{#1}%
\DeclareFieldFormat[book]{title}{{\textit{#1}}}%
\DeclareFieldFormat{doi}{%
  \ifhyperref
    {\href{http://dx.doi.org/#1}{\nolinkurl{doi:#1}}}%
    {\nolinkurl{doi:#1}}}%
\renewbibmacro*{in:}{}%
\DeclareFieldFormat{isbn}{ISBN #1}%
\DeclareFieldFormat{pages}{#1}%
\DeclareFieldFormat{url}{\url{#1}}%
\DeclareFieldFormat{urldate}{\mkbibparens{#1}}%
\renewcommand*{\cite}[2][]{\textcite[#1]{#2}}%

\newif\ifstarttheorem
\declaretheoremstyle[%
  spaceabove=0.5em,
  spacebelow=0.5em,
  headfont=\bfseries\global\starttheoremtrue,%
  notefont=\bfseries,%
  notebraces={(}{)},
  headpunct={},
  bodyfont=\normalfont,
  postheadspace=\newline%
]{myMainStyle}
\declaretheorem[style=myMainStyle]{definition}%
\declaretheorem[style=myMainStyle]{proposition}

\declaretheorem[style=myMainStyle]{theorem}
\declaretheorem[style=myMainStyle]{corollary}
\declaretheorem[style=myMainStyle]{remark}
\declaretheorem[style=myMainStyle]{example}
\declaretheorem[style=myMainStyle]{algorithm}

\makeatletter
\preto\itemize{%
  \if@inlabel
    \ifstarttheorem
      \mbox{}\par\nobreak\vskip\glueexpr-\parskip-\baselineskip+0.25em\relax\hrule\@height\z@
    \fi%
  \fi%
  \global\starttheoremfalse%
 \def\tempa{proof}%
 \ifx\tempa\mycurrenvir
    \ifstarttheorem
      \mbox{}\par\nobreak\vskip\glueexpr-\parskip-\baselineskip+0.25em\relax\hrule\@height\z@
    \fi%
 \fi%
 \global\starttheoremfalse%
}
\preto\enditemize{\global\starttheoremfalse}
\makeatother

\makeatletter
\preto\enumerate{%
  \if@inlabel
    \ifstarttheorem
      \mbox{}\par\nobreak\vskip\glueexpr-\parskip-\baselineskip+0.25em\relax\hrule\@height\z@
    \fi%
  \fi%
  \global\starttheoremfalse%
 \def\tempa{proof}%
 \ifx\tempa\mycurrenvir
    \ifstarttheorem
      \mbox{}\par\nobreak\vskip\glueexpr-\parskip-\baselineskip+0.25em\relax\hrule\@height\z@
    \fi%
 \fi%
 \global\starttheoremfalse%
}
\preto\endenumerate{\global\starttheoremfalse}
\makeatother

\newcommand{\ou}[3]{%
  \mathrel{%
    \vcenter{\offinterlineskip
      \ialign{##\cr$#1$\cr\noalign{\kern-#3}$#2$\cr}%
    }%
  }%
}
\newcommand*{\omu}[3]{\underset{#3}{\overset{#1}{#2}}}
\newcommand*{\T}{^{\top}}
\renewcommand*{\i}{{-1}}

\newcommand*{\deq}{\omu{\text{\tiny{d}}}{=}{}}

\newcommand*{\IN}{\mathbb{N}}

\newcommand*{\IR}{\mathbb{R}}

\newcommand*{\Bern}{\operatorname{Bern}}

\newcommand*{\LN}{\operatorname{LN}}

\newcommand*{\U}{\operatorname{U}}
\newcommand*{\B}{\operatorname{B}}

\newcommand*{\N}{\operatorname{N}}

\newcommand*{\I}{\mathbbm{1}}
\newcommand*{\rd}{\mathrm{d}}

\renewcommand*{\P}{\mathbb{P}}
\newcommand*{\E}{\mathbb{E}}

\newcommand*{\conv}{\operatorname{conv}}

\newcommand*{\Var}{\operatorname{Var}}
\newcommand*{\Cov}{\operatorname{Cov}}

\newcommand*{\R}{\textsf{R}}

\newcommand{\bbb}{\bm{b}}
\newcommand{\bc}{\bm{c}}

\newcommand{\bZ}{\bm{Z}}

\newcommand{\bW}{\bm{W}}

\newcommand{\bU}{\bm{U}}

\newcommand{\bV}{\bm{V}}

\begin{document}
\thispagestyle{plain}
\begin{center}
  {\Large COMPATIBILITY AND ATTAINABILITY OF MATRICES OF CORRELATION-BASED MEASURES OF CONCORDANCE\par}%
  \bigskip\smallskip
  BY\par %
  \bigskip\smallskip %
  {\Large{\scshape Marius Hofert and Takaaki Koike}\par}%
\end{center}
\par\smallskip
\begin{abstract}
  Measures of concordance have been widely used in insurance and risk management
  to summarize non-linear dependence among risks modeled by random variables,
  which Pearson's correlation coefficient cannot capture. However, popular
  measures of concordance, such as Spearman's rho and Blomqvist's beta, appear
  as classical correlations of transformed random variables.  We characterize a
  whole class of such concordance measures arising from correlations of
  transformed random variables, which includes Spearman's rho, Blomqvist's beta
  and van der Waerden's coefficient as special cases.  Compatibility and
  attainability of square matrices with entries given by such measures are
  studied, that is, whether a given square matrix of such measures of
  concordance can be realized for some random vector and how such a random
  vector can be constructed. Compatibility and attainability of block matrices
  and hierarchical matrices are also studied due to their practical importance
  in insurance and risk management. In particular, a subclass of attainable
  block Spearman's rho matrices is proposed to compensate for the drawback that
  Spearman's rho matrices are in general not attainable for dimensions larger
  than four.  Another result concerns a novel analytical form of the Cholesky
  factor of block matrices which allows one, for example, to construct random
  vectors with given block matrices of van der Waerden's coefficient.
\end{abstract}
\begin{center}%
{\scshape Keywords}%
\end{center}%
Transformed rank correlation coefficients, matrices of pairwise measures of
concordance, compatibility, attainability, copula models, Cholesky factor, block
correlation matrices, hierarchical matrices.

\section{Introduction}
Since the work of \cite{embrechtsmcneilstraumann1999}, copulas have been widely
adopted in insurance and risk management to quantify dependence between
continuously distributed random variables; see
\cite{genestgendronbourdeaubrien2009}. To summarize the dependence captured by
the copula by a single number, measures of concordance are frequently used. For
more than two random variables, multivariate measures of concordance
exist but are typically not unique extensions of their
bivariate counterparts to higher dimensions; see \cite{joe1990multivariate},
\cite[Chapter~10]{jaworskidurantehaerdlerychlik2010} and references
therein. Similar to the notion of correlation, matrices of (pairwise) measures
of concordance have recently become of interest; see, for
example, \cite{embrechtshofertwang2016} (motivated from an application in
insurance practice) for the notion of tail dependence. For such matrices of
measures of concordance, we study their compatiblity and
attainability. \emph{Compatibility} concerns whether a
given square matrix can be realized as a matrix of measures of concordance
of some random vector, and \emph{attainability} asks how
to construct such a random vector.  These notions are important in insurance and
risk management practice since the entries of matrices of pairwise measures of
concordance are often provided as estimates from real data (if available) or
from expert opinion based on scenarios (if no data is available or not directly
usable to estimate the entries). A primary issue is then to determine whether
the given matrix is admissible as a matrix of pairwise measures of concordance
and, if so, an appropriate model is built on the assumption of admissibility of
the given matrix; see \cite{embrechtsmcneilstraumann2002} and
\cite[Section~8.4]{mcneilfreyembrechts2015} for a discussion on compatibility
and attainability.

Note that compatibility is clear for Pearson's correlation coefficient since a
given $[-1,1]$-valued symmetric matrix $P$ is compatible if and only if it is
positive semi-definite and has diagonal entries equal to one. Also,
attainability is clear for Pearson's correlation coefficient since any symmetric
and positive semi-definite matrix $P$ with ones on the diagonal is attainable by
$\bm{X}=A\bm{Z}$ where $\bm{Z}$ is a random vector of independent standard
normal distributions and $A$ is the \emph{Cholesky factor} of $P$, that is, a
lower triangular matrix with non-negative diagonal entries and such that
$P=AA\T$.

Although compatibility and attainability of correlation matrices are thus
trivial, the limitations of Pearson's correlation coefficient as a dependence
measure are well known; see \cite{embrechtsmcneilstraumann2002}. Measures of
concordance in the sense of \cite{scarsini1984} are a remedy for some of the
pitfalls of the correlation coefficient and are thus considered more suitable to
summarize dependence between risks. Interestingly, such measures can also arise
as correlations, Spearman's rho, Blomqvist's beta and van der Waerden's
coefficient being prominent examples, all being correlations of transforms of
the underlying random variables.

Block matrices of measures of concordance naturally emerge if the risks of
interest are grouped based on business line, industry, country, etc.; see, for
example, \cite{huang2010correlation}. Hierarchical matrices are important
special cases of block matrices where a measure of concordance between two
variables is determined by an underlying hierarchical tree structure; see
\cite{hofertscherer2011} for an application to CDO pricing. Since such matrices are
typically high-dimensional, it is practically important to reduce the dimension
to solve compatibility and attainability problems in this case.

In this paper, we answer the following open questions, which naturally arise
regarding compatibility and attainability of transformed correlation
coefficients:
\begin{enumerate}
\item Are there more concordance measures which arise as correlations, and if
  so, how can they be characterized or constructed? (See Section~\ref{sec:cor:based:measures})
\item What about the compatibility and attainability of matrices of
  such measures? (See Section~\ref{sec:compat})
\item Can compatibility and attainability be reduced to lower
  dimensional problems if a matrix has block structure? (See
  Section~\ref{sec:block:mat})
\end{enumerate}

\section{Correlation-based measures of concordance}\label{sec:cor:based:measures}
We start by considering the bivariate case. To this end, let $X_1\sim F_1$ and
$X_2\sim F_2$ be two continuously distributed random variables with a
  unique copula $C$ such that $(U_1,U_2)=(F_1(X_1),F_2(X_2)) \sim C$.  The
  measures of concordance of $(X_1,X_2)$ we consider are of the form
  \begin{align}
    \kappa_{g_1,g_2}(X_1,X_2) = \rho\bigl(g_1(F_1(X_1)),g_2(F_2(X_2))\bigr), \label{transformed:g1g2cor}
  \end{align}
  where $g_1: [0,1]\rightarrow \IR$ and $g_2: [0,1]\rightarrow \IR$ are
  measurable functions, and $\rho$ is Pearson's correlation coefficient. Since
  \eqref{transformed:g1g2cor} depends only on the copula of $(X_1,X_2)$, we also
  denote it by $\kappa_{g_1,g_2}(C)=\rho(g_1(U_1),g_2(U_2))$ for
  $(U_1,U_2)\sim C$.  We are interested in conditions on $g_1$ and $g_2$ under
  which \eqref{transformed:g1g2cor} is a measure of concordance in the sense of
  \cite{scarsini1984}.  The following proposition provides a necessary condition
  on $g_1$ and $g_2$.

  \begin{proposition}[Monotonicity of $g_1$ and $g_2$]\label{prop:monotonicity:g}
  Suppose $g_1: [0,1]\rightarrow \IR$ and $g_2: [0,1]\rightarrow \IR$ are continuous functions. If $\kappa_{g_1,g_2}$ defined in \eqref{transformed:g1g2cor} is a measure of concordance, then $g_1$ and $g_2$ must be both increasing or both decreasing, that is,
  \begin{align*}
  (g_1(u')-g_1(u)) (g_2(v')-g_2(v))\geq 0,
  \end{align*}
for any $0\le u<u'\le 1$ and $0\le v<v'\le 1$.
  \end{proposition}
  \begin{proof}
    For $0\le u<u'\le 1$ and $0\le v<v'\le 1$, there exists a sufficiently large $N \in \IN$ and indices $i,i',j,j' \in \{1,\dots,N\}$ such that
  \begin{align*}
\frac{i-1}{N}<u \leq \frac{i}{N},\quad
\frac{i'-1}{N}<u' \leq \frac{i'}{N},\quad
\frac{j-1}{N}<v \leq \frac{j}{N},\quad
\frac{j'-1}{N}<v' \leq \frac{j'}{N}
  \end{align*}
  with $(\frac{i-1}{N},\frac{i}{N}]\cap (\frac{i'-1}{N},\frac{i'}{N}]=\emptyset$ and $(\frac{j-1}{N},\frac{j}{N}]\cap (\frac{j'-1}{N},\frac{j'}{N}]=\emptyset$.
  Let
    \begin{align*}
    \delta(x,y) &=\begin{cases}
      0, & (x,y) \in (\frac{i-1}{N},\frac{i}{N}]\times (\frac{j-1}{N},\frac{j}{N}] \cup (\frac{i'-1}{N},\frac{i'}{N}]\times (\frac{j'-1}{N},\frac{j'}{N}], \\
      2, &  (x,y) \in (\frac{i-1}{N},\frac{i}{N}]\times (\frac{j'-1}{N},\frac{j'}{N}] \cup (\frac{i'-1}{N},\frac{i'}{N}]\times (\frac{j-1}{N},\frac{j}{N}],\\
      1, & \text{otherwise},
    \end{cases}
    \end{align*}
    and
    \begin{align*}
    \tilde \delta(x,y) &=\begin{cases}
      2, &  (x,y)\in (\frac{i-1}{N},\frac{i}{N}]\times (\frac{j-1}{N},\frac{j}{N}] \cup (\frac{i'-1}{N},\frac{i'}{N}]\times (\frac{j'-1}{N},\frac{j'}{N}],\\
      0, &  (x,y) \in (\frac{i-1}{N},\frac{i}{N}]\times (\frac{j'-1}{N},\frac{j'}{N}] \cup (\frac{i'-1}{N},\frac{i'}{N}]\times (\frac{j-1}{N},\frac{j}{N}],\\
      1, & \text{otherwise},
    \end{cases}
  \end{align*}
  and let $Q_N$ and $\tilde Q_N$ be checkerboard copulas having densities $\delta$ and $\tilde \delta$, respectively; see \cite{carley2002new}.
  Then $Q_N \preceq \tilde Q_N$ (in concordance order), since for any supermodular function $\psi$ on $(0,1)^2$,
  \begin{align*}
  \int\psi\,\rd \tilde Q_N -   \int\psi\,\rd Q_N &= \int \psi \,\rd (\tilde Q_N-Q_N)\\
                                                 &=2 \int_{(0,1/N)^2}(\psi(i'-1+s,\ j'-1+t)+\psi(i-1+s,\ j-1+t)\\
                                                 &\phantom{=2 \int_{(0,1/N)^2}(} -\psi(i'-1+s,\ j-1+t)-\psi(i-1+s,\ j'-1+t))\,\rd s\,\rd t\geq 0,
  \end{align*}
  where the last inequality follows since the integrand is nonnegative for any $(s,t)\in (0,1/N)^2$ by supermodularity of $\psi$.
  The inequality $\int\psi\,\rd \tilde Q_N -   \int\psi\,\rd Q_N\geq 0$ for any supermodular function $\psi$ implies $Q_N \preceq \tilde Q_N$; see \cite{tchen1980inequalities} and \cite{muellerscarsini2000}.
  Since $\kappa_{g_1,g_2}$ is a measure of concordance, coherence of $\kappa_{g_1,g_2}$ implies that $\kappa_{g_1,g_2}(Q_N)\leq \kappa_{g_1,g_2}(\tilde Q_N)$, that is,
  \begin{align*}
  0\le \kappa_{g_1,g_2}(\tilde Q_N)-  \kappa_{g_1,g_2}(Q_N)&=\int_{(0,1)^2}g_1(U_1)g_2(U_2)\rd (\tilde Q_N - Q_N)\\ %
  &= 2 \int_{(0,1/N)^2}(g_1(i'-1+s)g_2(j'-1+t)+g_1(i-1+s)g_2(j-1+t)\\
  &\phantom{=2 \int_{(0,1/N)^2}(}-g_1(i'-1+s)g_2(j-1+t)-g_1(i-1+s)g_2(j'-1+t))\,\rd s\,\rd t;
  \end{align*}
  see \cite{scarsini1984} for the coherence axiom of a measure of concordance.
 Since $g_1$ and $g_2$ are continuous, apply the intermediate value theorem and let $N \rightarrow \infty$ to obtain that
  \begin{align*}
   g_1(u')g_2(v')+g_1(u)g_2(v)-g_1(u')g_2(v)-g_1(u)g_2(v')= (g_1(u')-g_1(u))(g_2(v')-g_2(v)) \geq 0,
  \end{align*}
  which shows that $g_1$ and $g_2$ are both increasing or both decreasing.
  \end{proof}
  By Proposition~\ref{prop:monotonicity:g}, $g_1$ and $g_2$ must be monotone
  with each other so that $\kappa_{g_1,g_2}$ is a measure of concordance.
  Therefore, it is reasonable to assume that $g_1$ and $g_2$ are both increasing
  functions on $[0,1]$ since, if both are decreasing, then
  $\kappa_{g_1,g_2}=\kappa_{\tilde g_1,\tilde g_2}$ for the increasing functions
  $\tilde g_1=1-g_1$ and $\tilde g_2=1-g_2$ by invariance of the correlation
  coefficient under linear transformations. If we relax the assumption of continuity
  of $g_1$ and $g_2$ to left-continuity, then $g_1$ and $g_2$ are quantiles of
  some distributions, say, $G_1$ and $G_2$. Recall that for a distribution function
  $G:\IR \rightarrow [0,1]$, its \emph{quantile function} is defined by
  \begin{align*}
    G^\i(p)=\inf\{x\in\IR:G(x)\ge p\},\quad p\in(0,1);
  \end{align*}
  see \cite{embrechtshofert2013c}.  By taking $g_1=G_1^\i$ and $g_2=G_2^\i$, we now
  define the $(G_1,G_2)$-\emph{transformed rank correlation coefficient} as
  follows.
  \begin{definition}[$(G_1,G_2)$-transformed rank correlation coefficient]
  Let $G_1 $ and $G_2$ be two distribution functions with quantile functions $G_1^\i$ and $G_2^\i$, respectively.
  For a random vector $(X_1,X_2)$ with continuous margins $F_1$ and $F_2$, the \emph{$(G_1,G_2)$-transformed rank correlation coefficient} is defined by
  \begin{align}
  \kappa_{G_1,G_2}(X_1,X_2) = \rho\bigl(G^\i_1(F_1(X_1)),G^\i_2(F_2(X_2))\bigr). \label{transformed:cor}
\end{align}
If $G_1=G_2=G$, $\kappa_{G,G}$ is denoted by $\kappa_{G}$ and
referred to as \emph{$G$-transformed rank correlation coefficient}.
  \end{definition}

\begin{example}[Known special cases of $\kappa_{G_1,G_2}$]\label{ex:spearman:beta:waerden}
  \begin{enumerate}
  \item If $G$ is the distribution function of the standard uniform distribution $\U(0,1)$, we obtain
    \begin{align*}
      \kappa_{G}(X_1,X_2) = \rho(F_1(X_1),F_2(X_2))
    \end{align*}
    from~\eqref{transformed:cor}.
    This is known as \emph{Spearman's rho} $\rho_{\text{S}}$; see \cite{spearman1904general}.
  \item\label{ex:spearman:beta:waerden:b} If $G$ is the distribution function of the symmetric Bernoulli distribution $\Bern(1/2)$, that is,
    \begin{align*}
      G(x)=\begin{cases}
        0, &  x < 0, \\
        1/2,  & 0\leq x < 1, \\
        1,  & x \geq 1,
      \end{cases}
    \end{align*}
    then $G^\i(p)=\I_{\{1/2<p\le 1\}}$ for $p\in (0,1)$. Therefore, since
    $U_j=F_j(X_j) \sim \U(0,1)$, $j=1,2$, \eqref{transformed:cor} is the
    correlation coefficient of $B_j = G_j^\i(F_j(X_j))\sim \Bern(1/2)$,
    $j=1,2$. If $C$ denotes the distribution function of $(U_1,U_2)$ and
    $G_1=G_2=G$, then
    \begin{align*}
      \kappa_{G}(X_1,X_2) &= \frac{\E(B_1B_2)-\E(B_1)\E(B_2)}{\sqrt{\Var(B_1)\Var(B_2)}}=\frac{\P(U_1>1/2, U_2>1/2)-1/4}{1/4}\notag\\
                              &= 4 \P(U_1>1/2, U_2>1/2) - 1 = 4 (1-1/2-1/2+C(1/2,1/2)) - 1\notag\\
                              &= 4C(1/2,1/2)-1
    \end{align*}
    which equals \emph{Blomqvist's beta} $\beta$; see \cite{blomqvist1950measure}.
    Note that Blomqvist's beta is also known as \emph{median correlation coefficient}.
  \item If $G$ is the distribution function $\Phi$ of the standard normal distribution $\N(0,1)$, then
    \begin{align*}
      \kappa_{G}(X_1,X_2) = \rho\bigl(\Phi^{-1}(F_1(X_1)),\Phi^{-1}(F_2(X_2))\bigr)
    \end{align*}
    which equals \emph{van der Waerden's coefficient} $\zeta$; see, for example, \cite{sidak1999theory}.
    It is also known as \emph{normal score correlation}.
  \end{enumerate}
\end{example}

The first question in the introduction is natural: For which distributions
$G_1,G_2$ does the $G_1,G_2$-transformed correlation $\kappa_{G_1,G_2}$ lead to
a measure of concordance in the sense of \cite{scarsini1984}? Before answering
it, consider the following example in the spirit of
\cite{embrechtsmcneilstraumann2002}; another example of this type is the
correlation bounds of Bernoulli random variables; see Example~\ref{ex:bernoulli:G}. Both examples show that $G_1$ and $G_2$ cannot
be chosen arbitrarily.

\begin{example}[Log-normal $G_1,G_2$-functions]\label{ex:lognormal:G}
  For $j=1,2$, let $\sigma_j>0$ and $G_j$ be the distribution function of the log-normal distribution $\LN(0,\sigma_j)$.
  Since $\kappa_{G_1,G_2}$ is the correlation coefficient of the
  random vector $(G_1^{-1}(U_1),G_2^{-1}(U_2))$ with
  $(U_1,U_2)=(F_1(X_1),F_2(X_2))$, its minimal and maximal values are attained when
  $(X_1,X_2)$ has copula $C=W$ and $C=M$, respectively, where $W(u_1,u_2)=\max\{u_1+u_2-1,0\}$ is the
  countermonotone and $M(u_1,u_2)=\min\{u_1,u_2\}$ is the comonotone copula. For different pairs
  of $(\sigma_1,\sigma_2)$, the minimal and maximal $(G_1,G_2)$-transformed rank correlation coefficients are shown in
  Figure~\ref{Figure: min and max correlations of Lognormal}
  as correlation coefficients of $\LN(0,\sigma_1)$ and $\LN(0,\sigma_2)$.
  \begin{figure}[htbp]
    \begin{minipage}{0.5\hsize}
      \begin{center}
        \includegraphics[width=70mm]{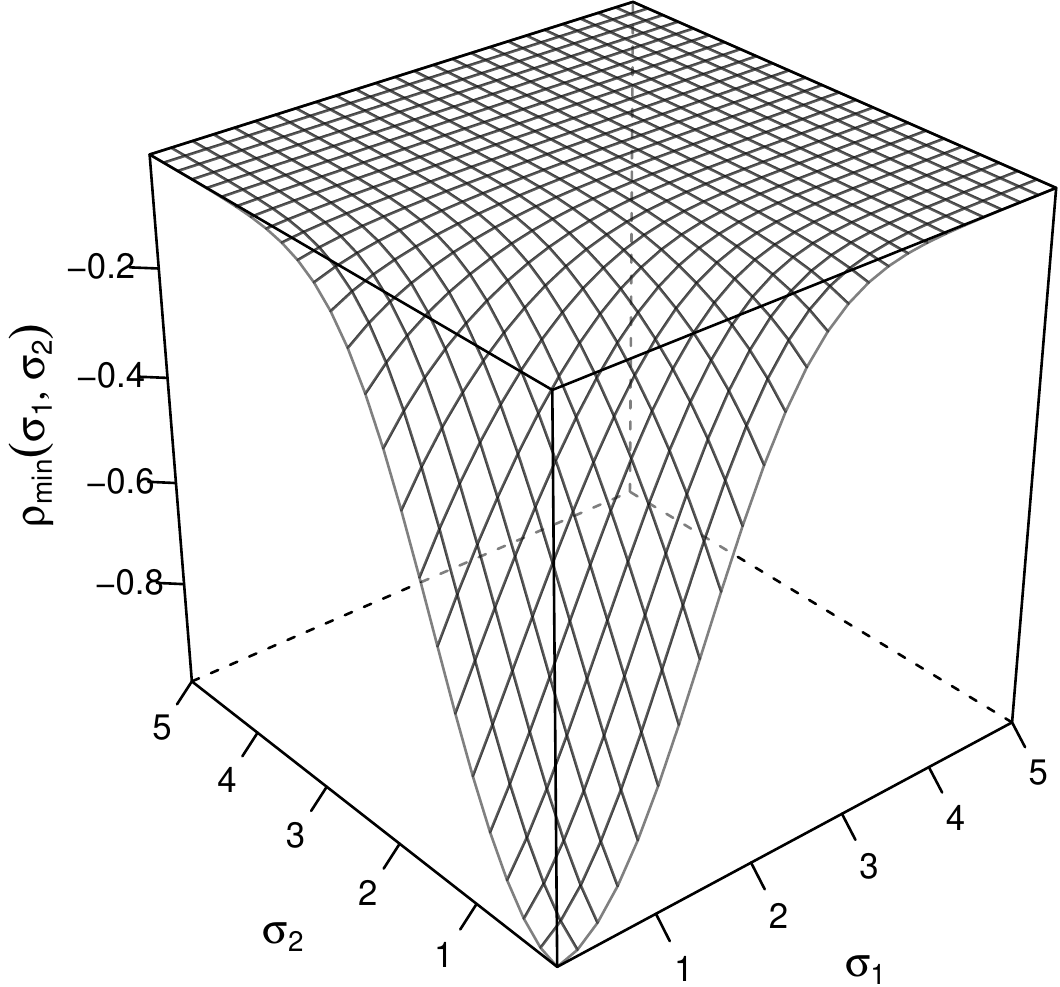}
      \end{center}
    \end{minipage}
    \begin{minipage}{0.5\hsize}
      \begin{center}
        \includegraphics[width=70mm]{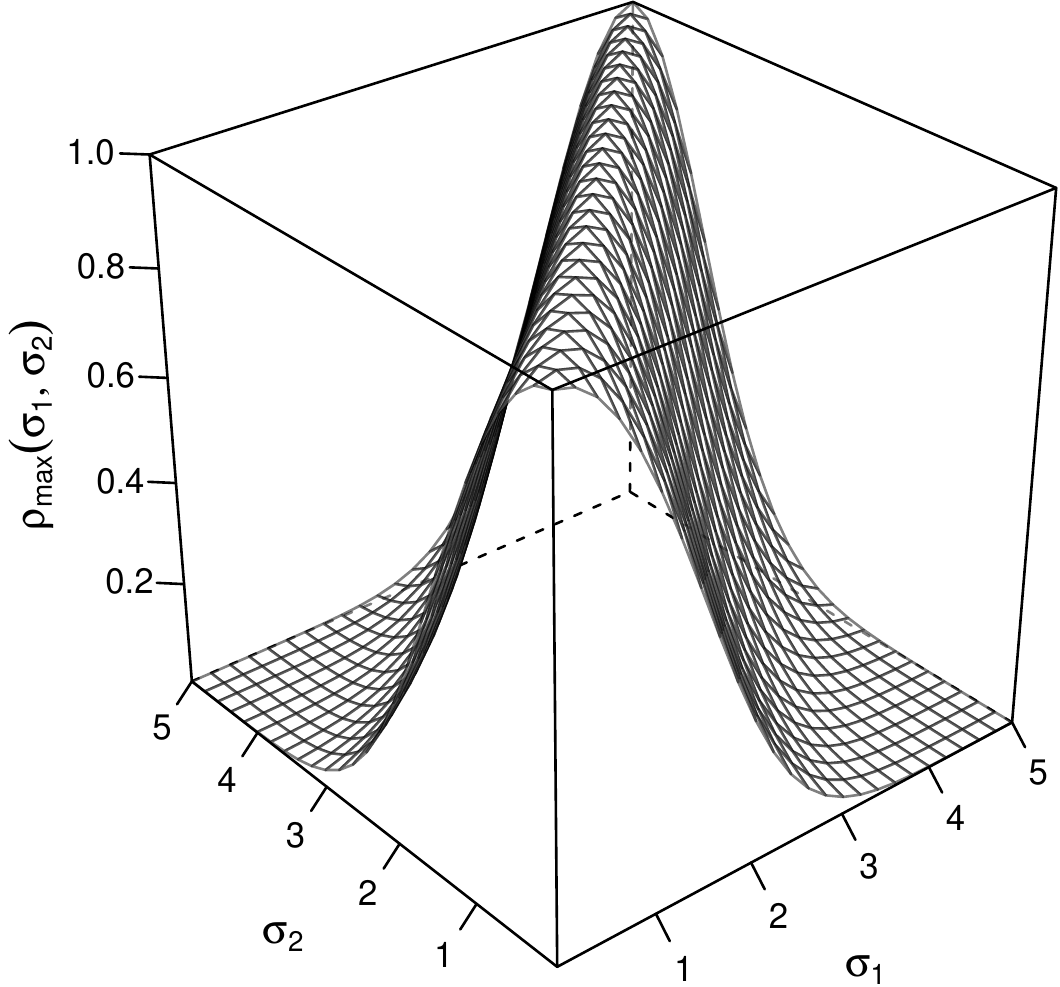}
      \end{center}
    \end{minipage}
    \caption{Minimal (left) and maximal (right) correlations attained by the
      $(G_1,G_2)$-transformed rank correlation coefficient $\kappa_{G_1,G_2}$
      where $G_j$ is the distribution function of $\LN(0,\sigma_j)$, $j=1,2$.}
    \label{Figure: min and max correlations of Lognormal}
  \end{figure}
  The left-hand side of this figure shows that $\kappa_{G_1,G_2}=-1$ is not
  attained for any $\sigma_1,\sigma_2>0$ and the right-hand side shows that
  $\kappa_{G_1,G_2}=1$ is not attained unless $\sigma_1=\sigma_2$.
  Consequently, if $G_1,G_2$ are taken to be log-normal distribution functions,
  $\kappa_{G_1,G_2}$ cannot be a measure of concordance since the range axiom is
  violated; see \cite{scarsini1984}.
\end{example}

The main result of this section is the following, which provides necessary and
sufficient conditions for a transformed rank correlation coefficient to be a
measure of concordance in the sense of \cite{scarsini1984}. Recall
  that two distributions are \emph{of the same type} if one is a location-scale
  transform of the other.
\begin{theorem}[Necessary and sufficient conditions for transformed rank correlations to be measures of concordance]\label{thm:G1:G2}
  Let $G_1,G_2$ be distribution functions. The $(G_1,G_2)$-transformed rank
  correlation coefficient $\kappa_{G_1,G_2}$ in \eqref{transformed:cor} is a
  measure of concordance if and only if both $G_1$ and $G_2$ are of the same
    type as some non-degenerate symmetric distribution $G$ with finite second
  moment.
\end{theorem}
\begin{proof}
  Let $(X_1,X_2)\sim H$ with copula $C$ and continuous margins $F_1,F_2$. Then
  $(U_1,U_2) = (F_1(X_1),F_2(X_2)) \sim C$ so that
  $(Y_1,Y_2)=(G_1^\i(U_1),G_2^\i(U_2))$ has copula $C$ and marginal distribution
  functions $G_1,G_2$. The transformed rank correlation coefficient
  $\kappa_{G_1,G_2}(X_1,X_2)$ in \eqref{transformed:cor} can then be written as
  $\kappa_{G_1,G_2}(X_1,X_2)=\rho(Y_1,Y_2)$.

  Consider necessity. %
  If either of $G_1$ and $G_2$ is degenerate, then $\rho(Y_1,Y_2)$ is not
  well-defined, %
  which violates the domain axiom of a measure of concordance.
  Therefore, $G_1$ and $G_2$ must be non-degenerate. Next, if either of
  $\Var(Y_1)$ and $\Var(Y_2)$ is infinite, then $\rho(Y_1,Y_2)$ is not defined,
  which also violates the domain axiom. Thus, $G_1$ and $G_2$ must have finite
  second moments.
  For $j=1,2$, let $\mu_j = \E(Y_j)$ and $\sigma_j^2 = \Var(Y_j)<\infty$.  It is
  known that %
  $\rho(Y_1,Y_2)=-1$ if and only if $Y_2\deq-aY_1+b$ for some $a,b\in\IR$ with
  $a>0$ and $\rho(Y_1,Y_2)=1$ if and only if $Y_2\deq cY_1+d$ for some
  $c,d\in\IR$ with $c>0$.  Note that both distributional equalities must hold
  simultaneously so that $\kappa_{G_1,G_2}(X_1,X_2)=1$ when $(X_1,X_2)$ is
  comonotone and $\kappa_{G_1,G_2}(X_1,X_2)=-1$ when $(X_1,X_2)$ is
  countermonotone.  Since $\sigma_2^2=a^2\sigma_1^2=c^2\sigma_1^2$, $a,c>0$ and
  $\sigma_1\neq 0$, we have $a=c$. Furthermore, by taking expectations,
  $\mu_2=-c\mu_1+b$ and $\mu_2=c\mu_1+d$, which imply that $\mu_1=(b-d)/(2c)$
  and $\mu_2=(b+d)/2$.  Since $Y_2-b\deq -c Y_1 \deq d-Y_2$, adding constant
  $(b-d)/2$ to both hand sides yield $Y_2 - \mu_2 \deq \mu_2 - Y_2$.  This
  implies that $Y_2$ is symmetric about its mean $\mu_2$.  Similarly, $Y_1$ is
  shown to be symmetric about its mean $\mu_1$.  Finally, it follows from
  $Y_2\deq cY_1+d$ that $G_2(x)=G_1((x-d)/c)$ and thus
  $G_2^\i(u)=d+c\ G_1^\i(u)$, which concludes the proof of necessity.

Now consider sufficiency. %
If $G_1$ and $G_2$ are of the same type with some distribution $G$, then
$\kappa_{G_1,G_2}(C)=\kappa_{G,G}(C)=\kappa_{G}(C)$ for any copula $C$ since
correlation coefficient is invariant under positive linear transform; see
\cite{embrechtsmcneilstraumann2002}. Therefore, it suffices to verify the seven
axioms of a measure of concordance in \cite{scarsini1984} for $\kappa_{G}$ with
$G$ being a non-degenerate symmetric distribution with finite second moment.
  \begin{enumerate}
  \item \emph{Domain:} Since $G$ is non-degenerated with a finite second moment,
    $\rho(Y_1,Y_2)$ is well-defined for all continuously distributed $X_1,X_2$.
  \item \emph{Symmetry:} To show $\kappa_G(X_1,X_2)=\kappa_G(X_2,X_1)$, it suffices to show
    \begin{align*}
      \E(G^\i(U_1)G^\i(U_2))=\E(G^\i(U_2)G^\i(U_1))
    \end{align*}
    for any $C$ and $(U_1,U_2)\sim C$, but this is obvious by exchangeability of product.
  \item \emph{Coherence:} Let $C_1,C_2$ be copulas such that $C_1\preceq C_2$,
    that is, $C_1(u_1,u_2)\le C_2(u_1,u_2)$ for all $u_1,u_2\in[0,1]$.
    Then $\kappa_{G}(C_1)\le\kappa_{G}(C_2)$ follows immediately from
    the Hoeffding's lemma; see~\cite[Lemma~7.27]{mcneilfreyembrechts2015}.
  \item \emph{Range:} Since $\kappa_G(X_1,X_2)=\rho(Y_1,Y_2)$, we have $-1 \leq \kappa_G(X_1,X_2)\leq 1$. Moreover, since $G$ is symmetric, we have $Y_1-\E[Y_1]\deq \E[Y_2]-Y_2$.
  Together with $Y_1\deq Y_2$, the bounds $\kappa_G(X_1,X_2)=-1$ and $\kappa_G(X_1,X_2)=1$ are attainable when $(X_1,X_2)$ are countermonotone and comonotone, respectively.
  \item \emph{Independence:} When $X_1,X_2$ are independent, so are $Y_1,Y_2$ and thus $\kappa_{G}(X_1,X_2)=\rho(Y_1,Y_2)=0$.
  \item \emph{Change of sign:} Let $F_{-X_2}$ be the distribution of $-X_2$.
    Then it holds that $F_{-X_2}(-x_2)=\P(X_2 > x_2) = 1-F_2(x_2)$ and thus
    $F_{-X_2}(-X_2)=1-F_2(X_2)=1-U_2$.  Symmetry of $G$ implies that
    $G(y)=1-G(2\mu_2-y)$ for $y\in \IR$ and thus $G^\i(1-p) = 2\mu_2-G^\i(p)$
    for $p\in (0,1)$. Therefore,
    \begin{align*}
      \kappa(X_1,-X_2) &= \rho\bigl(G^\i(F_{X_1}(X_1)),G^\i(F_{-X_2}(-X_2))\bigr) = \rho(G^\i(U_1),G^\i(1-U_2))\\
      &= \rho(G^\i(U_1),2\mu_2-G^\i(U_2)) =\rho(G^\i(U_1),-G^\i(U_2)) \\
    &= -\rho(G_{1}^\i(U_1),G_{2}^\i(U_2))= -\kappa(X_1,X_2)
    \end{align*}
    by invariance and change of sign properties of correlation coefficient.
  \item \emph{Continuity:} Let $(X_{n1},X_{n2})\sim H_n$, $n\in\IN$, and
    $(X_1,X_2)\sim H$ all have continuous margins with $H_n$ converging
    pointwise to $H$ as $n\rightarrow \infty$.  Let $C_n$ denote the copula of
    $H_n$, $n\in\IN$, and $C$ the one of $H$.  Then
    $\lim_{n\rightarrow \infty}C_n = C$ pointwise.  Since
    $\kappa(X_{n1},X_{n2})$ and $\kappa(X_{1},X_{2})$ are correlation
    coefficients of $(Y_{n1},Y_{n2})$ and $(Y_1,Y_2)$ having the same marginal
    distribution $G$ and copulas $C_n$ and $C$, respectively, Hoeffding's lemma
    yields that
    \begin{align}\label{eq:expression:D4invariantMOC}
      \nonumber \lim_{n\rightarrow \infty}\kappa(X_{n1},X_{n2})&=\lim_{n\rightarrow \infty}\frac{1}{\sigma_1\sigma_2}\int_{\IR^2}(C_n(G(y_1),G(y_2))-G(y_1)G(y_2))\,\rd \lambda_2(y_1,y_2)\\
                                                      &=\frac{1}{\sigma_1\sigma_2}\int_{\IR^2}(C(G(y_1),G(y_2))-G(y_1)G(y_2))\,\rd \lambda_2(y_1,y_2)=\kappa(X_{1},X_{2}),
    \end{align}
    for the Lebesgue measure $\lambda_2$ on $\IR^2$, where the second equality is justified by the bounded convergence theorem
    since $C_n(G(y_1),G(y_2))-G(y_1)G(y_2)$ and $C(G(y_1),G(y_2))-G(y_1)G(y_2)$
    are all uniformly bounded.
  \end{enumerate}
\end{proof}

As seen in the proof of Theorem~\ref{thm:G1:G2}, if $\kappa_{G_1,G_2}$ is a
measure of concordance, then it must be written by $\kappa_{G}$ for some
distribution $G$ which is of the same type with $G_1$ and $G_2$.  In what
follows, we thus focus on $G$-transformed rank correlation coefficients for
which we assume that $G_1 = G_2$.

\begin{remark}[Connection to $D_4$-invariant measures of concordance]\label{rem:D4inv}
  From \eqref{eq:expression:D4invariantMOC} it turns out that
  $(G_1,G_2)$-transformed rank correlations $\kappa_{G_1,G_2}$ form a subclass
  of $D_4$-invariant measures of concordance as proposed by
  \cite{edwards2005measures}. A measure $\nu$ on $(0,1)^2$ is called
  \emph{$D_4$-invariant} if it is invariant under transpositions $(x,y)\mapsto (y,x)$
  and partial reflections $(x,y)\mapsto (1-x,y)$. For such
  measures $\nu$, \cite{edwards2005measures} show that the functional
  \begin{align}\label{eq:D4invariant:MOC}
    C \mapsto \frac{\int_{(0,1)^2}(C-\Pi)\,\rd \nu}{\int_{(0,1)^2}(M-\Pi)\,\rd \nu}
  \end{align}
  is a measure of concordance, where $M$ is the comonotonic copula and $\Pi$ is the independence copula. When $G_1$ and $G_2$ are symmetric, the
  pushforward Lebesgue measure $\lambda_{G_1,G_2}$ is $D_4$-invariant and the
  corresponding measure \eqref{eq:D4invariant:MOC} yields our
  $(G_1,G_2)$-transformed rank correlation \eqref{transformed:cor}.
  Consequently, the sufficiency part of the proof of Theorem~\ref{thm:G1:G2}
  follows from \cite[Theorem 0.6]{edwards2005measures}.
\end{remark}

According to Theorem~\ref{thm:G1:G2}, we call a distribution function $G$
\emph{concordance inducing} if it is non-degenerate, symmetric and has
finite second moment. Examples of such distributions include normal, Student's
$t$ with degrees of freedom $\nu>2$, continuous and discrete uniform
distributions, Laplace and logistic distributions.  The following example shows
that Bernoulli distribution $\Bern(p)$ is concordance inducing if and only if
they are symmetric, that is, $p=1/2$.

\begin{example}[Bernoulli $G$-function]\label{ex:bernoulli:G}
  For $j=1,2$, let $p_j \in [0,1]$ and $G_j$ be the distribution of
  $Y_j \sim \Bern(p_j)$.  As discussed in Example~\ref{ex:lognormal:G},
  $\kappa_{G_1,G_2}(X_1,X_2)=\rho(Y_1,Y_2)$ and its minimal and maximal values
  are attained when $C=W$ and $C=M$, respectively. Figure~\ref{Figure: min and
    max correlations of Bernoulli} illustrates the minimal (left-hand side) and
  maximal (right-hand side) $(G_1,G_2)$-transformed rank correlation
  coefficients as correlations of $\Bern(p_1)$ and $\Bern(p_2)$ for different
  pairs of $(p_1,p_2)$.
  \begin{figure}[htbp]
    \begin{minipage}{0.5\hsize}
      \begin{center}
        \includegraphics[width=70mm]{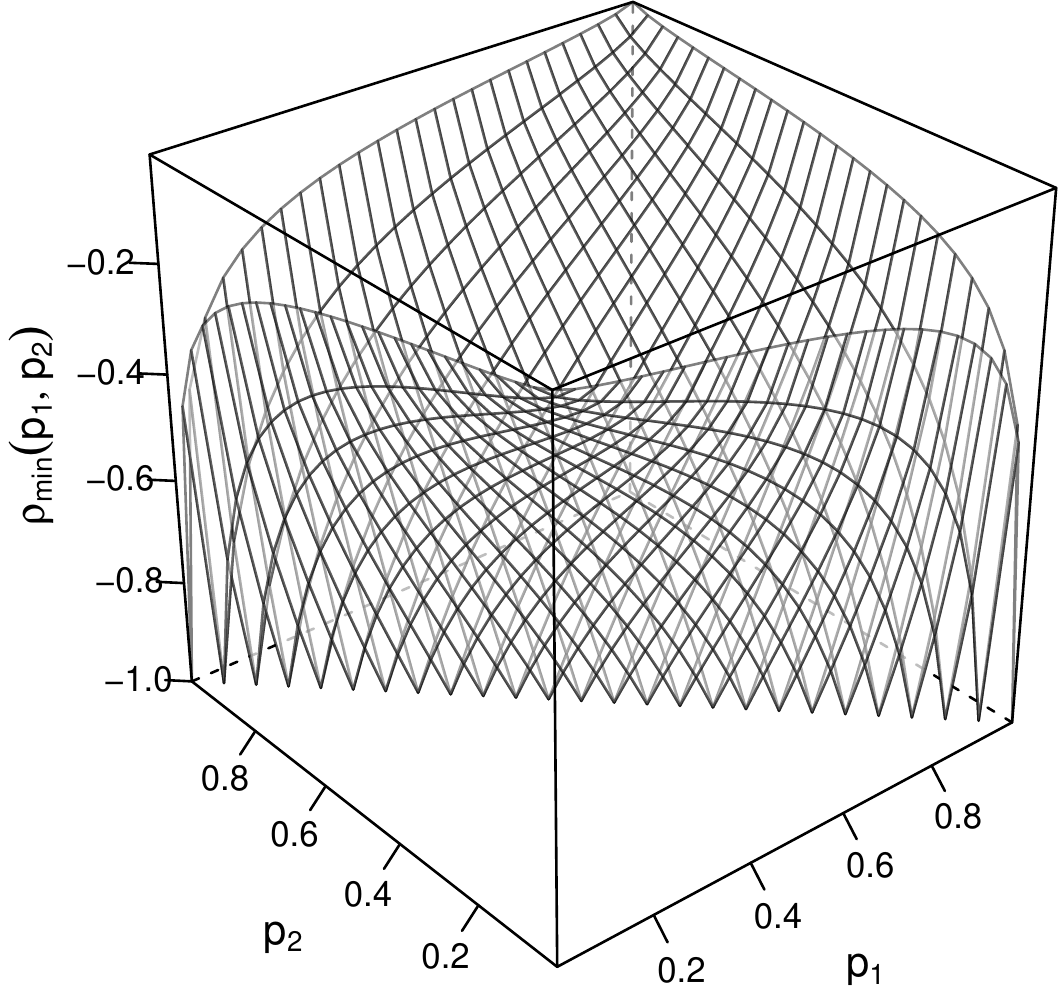}
      \end{center}
    \end{minipage}
    \begin{minipage}{0.5\hsize}
      \begin{center}
        \includegraphics[width=70mm]{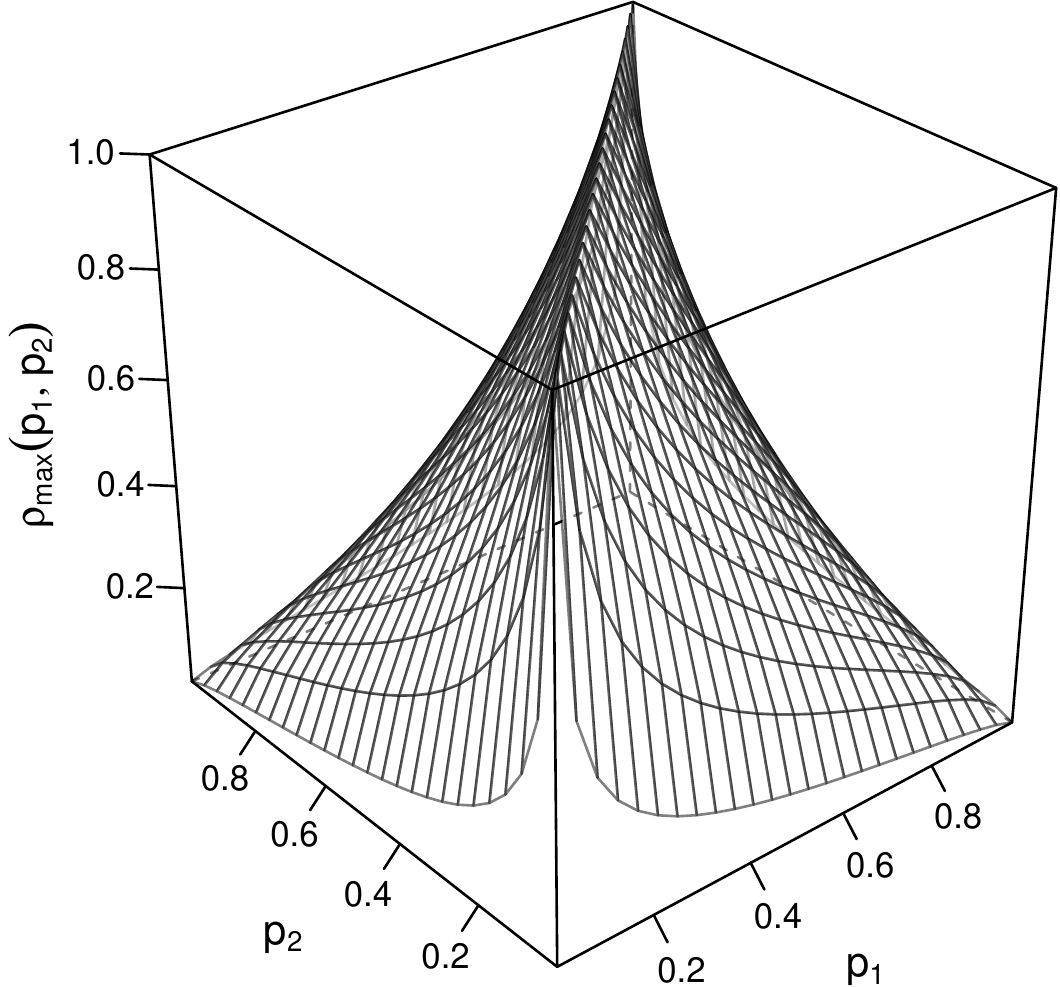}
      \end{center}
    \end{minipage}
    \caption{Minimal (left) and maximal (right) correlations attained by the
      $(G_1,G_2)$-transformed rank correlation coefficient $\kappa_{G_1,G_2}$
      where $G_j$ is the distribution function of $\B(1,p_j)$, $j=1,2$.}
    \label{Figure: min and max correlations of Bernoulli}
  \end{figure}
  The left-hand side of the figure indicates that $\kappa_{G_1,G_2}=-1$ if
  $p_1 = 1-p_2$ and this is the only case when $Y_1$ and $-Y_2$ are of the same type. The right-hand side shows that $\kappa_{G_1,G_2}=1$ if $p_1 = p_2$,
  and this is the only case when $Y_1$ and $Y_2$ have the same distribution.
  Since $\kappa_{G_1,G_2}$ must attain $-1$ and $1$ when $C=W$ and $C=M$,
  respectively, $\kappa_{G_1,G_2}$ is a measure of concordance only when
  $p_1=p_2=1/2$.  As a consequence, $\Bern(p)$ is concordance inducing if and
  only if $p=1/2$.
\end{example}

Note that due to the invariance of the correlation coefficient under strictly
increasing linear transforms, $\kappa_G$ is invariant under location-scale
transforms of $Y\sim G$. Therefore, if $G$ has bounded support, it may be
beneficial to standardize it so that its support is $[0,1]$.  Similarly, if $G$
is supported on $\IR$, one can still standardize $G$ to have zero mean and unit
variance without changing $\kappa_{G}$. Due to this property, one can see that
the quadrant correlation of \cite{mosteller2006some} studied in
\cite{raymaekersrousseuw2018robustcorrelation} coincides with Blomqvist's beta.

Uniqueness of $G$-function up to location-scale
transformations follows direcltly from \cite[Lemma~2.4]{edwards2004measures} or
\cite[Lemma~0.4]{edwards2005measures}.

\begin{proposition}[Uniqueness of $G$-functions]\label{prop:uniqueness:G}
Let $G$ and $G'$ be two continuous concordance-inducing functions.
If $\kappa_G(C)=\kappa_{G'}(C)$ for all 2-copulas, then $G$ and $G'$ are of the same type.
\end{proposition}
We end the section with a simple linear property of $\kappa_G$.

\begin{proposition}[Linearity of $\kappa_G$]\label{prop:linearity:kappaG}
  For $n \in \IN$, let $C_1,\dots,C_n$ be 2-copulas and
  $\alpha_1,\dots,\alpha_n$ be non-negative numbers such that
  $\alpha_1+\cdots+\alpha_n = 1$. Then
  \begin{align*}
    \kappa_G\biggl(\,\sum_{i=1}^{n}\alpha_i C_i \biggr)=\sum_{i=1}^{n}\alpha_i \kappa_G(C_i).
  \end{align*}
\end{proposition}
\begin{proof}
  As a mixture, $\sum_{i=1}^{n}\alpha_i C_i $ is a 2-copula from which
  the equation to prove is an immediate consequence of Hoeffding's lemma.
\end{proof}

\begin{remark}[Degree of $\kappa_G$]\label{remark:degree}
  For a general measure of concordance $\kappa$, \cite{edwards2009characterizations}
  defined the notion of a \emph{degree} as the maximum degree of the polynomial $t\mapsto \kappa(tC_1+(1-t)C_2)$, when it is the case, over any two copulas $C_1$ and $C_2$.
  Proposition~\ref{prop:linearity:kappaG} shows that $\kappa_G$ is a measure of
  concordance of degree one in this sense.
  Also note that the class of $G$-transformed rank
  correlation coefficients is a strict subclass of all measures of concordance
  of degree one since, for instance, Gini's coefficient is of degree one but
  cannot be represented as~\eqref{transformed:cor}; see
  Appendix~\ref{app:B}. Furthermore, there is no $G$-function that makes
  $\kappa_G$ Kendall's tau since the latter is a measure of concordance of
  degree two according to \cite{edwards2009characterizations}.  See Appendix~\ref{app:C}
  for a more detailed discussion on Kendall's tau.
\end{remark}

\section{Matrices of transformed rank correlation coefficients and their
  compatibility}\label{sec:compat}
Let $\bm{X}=(X_1,\dots,X_d)$ be a random vector with continuous margins
$F_1,\dots,F_d$ and copula $C$. We now consider matrices of (pairwise)
$G$-transformed rank correlation measures, that is, matrices
$P\in[-1,1]^{d\times d}$ with $(i,j)$th entry given by $\kappa_{G}(X_i,X_j)$. As
in Theorem~\ref{thm:G1:G2}, $G$ is set to be a distribution function of a
non-degenerate, symmetric distribution with finite second moment. We
call a given matrix $P \in [-1,1]^{d\times d}$ \emph{$\kappa_{G}$-compatible} if
there exists a $d$-random vector $\bm X$ such that $P=(\kappa_{G}(X_i,X_j))$. In
this section, we first study this \emph{compatibility problem} for the
transformed rank correlation coefficient \eqref{transformed:cor} in general and
then more specifically for Spearman's rho, Blomqvist's beta and van der
Waerden's coefficient. Note that an obvious necessary condition for a given
matrix $P$ to be $\kappa_{G}$-compatible is that it is a $[-1,1]^{d\times d}$
symmetric, positive semi-definite matrix with diagonal elements equal to $1$.

\subsection{A sufficient condition for compatibility of transformed rank correlation coefficients}
For a fixed concordance inducing function $G$, denote by $\mathcal{K}_{G}$ the set of all $\kappa_{G}$-compatible
matrices.  Since $\kappa_{G}(X_i,X_j)=\rho(Y_i,Y_j)$ with the notation as before,
$\mathcal{K}_{G}$ can be written as
\begin{align*}
  \mathcal{K}_{G}= \{\rho(\bm{Y})\,|\,\bm{Y} \in \mathcal F_{d}(G,\dots,G)\},
\end{align*}
where $\mathcal F_{d}(G,\dots,G)$ denotes the set of all $d$-dimensional random
vectors with all marginals equal to $G$. The following corollary follows
directly from Proposition~\ref{prop:linearity:kappaG}.
\begin{corollary}[Convexity of $\mathcal{K}_{G}$]\label{convexity of K_G set}
  $\mathcal{K}_{G}$ is a convex set for any concordance inducing function $G$.
\end{corollary}

Let
\begin{align*}
  \mathcal P^{\text{B}}_d(1/2)=\{\rho({\bm B}):{\bm B}=(B_1,\dots,B_d),\  B_j \sim \Bern(1/2),\ j=1,\dots,d\}
\end{align*}
be the set of all correlation matrices of $d$-dimensional random vectors whose
marginals are symmetric Bernoulli distributions. The following proposition provides a
sufficient condition for a given matrix to be $\kappa_{G}$-compatible.
\begin{proposition}[A sufficient condition for $\kappa_{G}$-compatibility]\label{prop:suff:cond:k:compatibility}
  For a concordance inducing function $G$, it holds that $\mathcal P^{\text{B}}_d(1/2) \subseteq \mathcal{K}_{G}$, that
  is, a given matrix $P \in[-1,1]^{d\times d}$ is $\kappa_{G}$-compatible if it
  is a correlation matrix of some random vector with $\Bern(1/2)$ margins.
\end{proposition}
\begin{proof}
  Fix $P \in  \mathcal P_d^{\text{B}}(1/2)$.
  Then there exist $B_1,\dots,B_d \sim \Bern(1/2)$ such that $\rho(\bm B)=P$ for $\bm B = (B_1,\dots,B_d)$.
  For $U\sim \U(0,1)$ independent of $\bm B$, define
  \begin{align*}
    V_j = B_j U + (1-B_j) (1-U),\quad j=1,\dots,d.
  \end{align*}
  Then $V_j \sim \U(0,1)$ and thus $Y_j=G^\i(V_j) \sim G$, $j=1,\dots,d$.
  Note that $Y_j=G^\i(U)$ if $B_j=1$ and $Y_j=G^\i(1-U)$ if $B_j=0$.
  Furthermore, since $G$ is concordance inducing,
  \begin{align*}
    \rho(G^\i(U),G^\i(U))=1\quad\text{and}\quad
    \rho(G^\i(U),G^\i(1-U))=-1.
  \end{align*}
  Consequently, for all $i,j\in\{1,\dots,d\}$,
  \begin{align*}
    \rho(Y_i,Y_j)  %
    &= \rho(G^\i(U),G^\i(U))\P(B_i=B_j) +  \rho(G^\i(U),G^\i(1-U))\P(B_i\neq B_j)\\
    &= \P(B_i=B_j)-\P(B_i\neq B_j) = 2\P(B_i=B_j)-1.
  \end{align*}
  Since
  \begin{align*}
    \P(B_i=B_j)&=\P(B_i=0,B_j=0)+\P(B_i=1,B_j=1)\\
               &=\P(1-B_i=1,1-B_j=1)+\E(B_iB_j)\\
               &=\E((1-B_i)(1-B_j))+\E(B_iB_j)=2\E(B_iB_j)\\
    &=\frac{\rho(B_i,B_j)+1}{2},
  \end{align*}
  we obtain
  \begin{align*}
  \rho(Y_i,Y_j)=2\frac{\rho(B_i,B_j)+1}{2}-1=\rho(B_i,B_j)
  \end{align*}
  and thus
  $P=\rho(\bm B)=\rho(\bm Y) \in \mathcal{K}_{G}$.
\end{proof}

Note that the construction $Y_j=G^\i (B_j U + (1-B_j) (1-U))$, $j=1,\dots,d$,
used in the proof of Proposition~\ref{prop:suff:cond:k:compatibility} was
utilized by \cite{huber2015multivariate} for the purpose of generating a
$d$-dimensional distribution with given margins $G$ and a correlation matrix $P$
where $P\in \mathcal P_d^{\text{B}}(1/2)$.

By Proposition~\ref{prop:suff:cond:k:compatibility}, a given matrix is found to
be $\kappa_{G}$-compatible if it belongs to $\mathcal P^{\text{B}}_d(1/2)$. The
relationship between $\mathcal{K}_{G}$ and $\mathcal P^{\text{B}}_d(1/2)$
depends on the $G$-function.  When $G$ is a symmetric Bernoulli distribution, it
holds that $\mathcal{K}_{G}=\mathcal P^{\text{B}}_d(1/2)$, whereas if $G$ is the
standard normal distribution function $\Phi$, then $\mathcal{K}_{G}$ coincides
with the set of all correlation matrices $\mathcal P_d$, which is strictly
larger than $\mathcal P^{\text{B}}_d(1/2)$; see
Proposition~\ref{prop:characterize:S:B:W}~\ref{prop:characterize:S:B:W:4} for
$\mathcal K_\Phi=\mathcal P_d$ and Section \ref{subsection:
  Bernoulli(1/2)-compatibility problem} for
$\mathcal P^{\text{B}}_d(1/2)\subset \mathcal P_d$.  As summarized by the
following corollary, $\mathcal P^{\text{B}}_d(1/2)$ and $\mathcal P_d$ are the
smallest and largest set of $\kappa_G$ compatible matrices for general $G$.

\begin{corollary}[Upper and lower bounds of $\mathcal K_G$]\label{corollary:upper:lower:bounds}
For any concordance inducing function $G$, the set of all $\kappa_G$-compatible matrices $\mathcal K_G$ satisfy $\mathcal P^{\text{B}}_d(1/2) \subseteq \mathcal K_G \subseteq \mathcal P_d$, and the upper and lower bounds are both attainable.
\end{corollary}

Note that the uniqueness of $G$ attaining bounds fails and possibly depend on $d$. For example, when $d\leq 9$, both of $G=\U(0,1)$ and $\Phi$ attain $\mathcal K_G=\mathcal P_d$; see Proposition~\ref{prop:characterize:S:B:W}~\ref{prop:characterize:S:B:W:1}, ~\ref{prop:characterize:S:B:W:2} and ~\ref{prop:characterize:S:B:W:4}.

Now we have found that the set $\mathcal P^{\text{B}}_d(1/2)$ plays important roles on $\kappa_G$-compatibility problem.
Natural questions regarding $\mathcal P^{\text{B}}_d(1/2)$ are how to check a given matrix belongs to $\mathcal P^{\text{B}}_d(1/2)$ and how large the set is in comparison to the set of all correlation matrices $\mathcal P_d$. These questions will be answered in Section~\ref{subsection: Bernoulli(1/2)-compatibility problem}.

\subsection{Characterizations of specific measures of concordance}
In this section, we study the three specific measures of concordance from
Example~\ref{ex:spearman:beta:waerden}, Spearman's rho, Blomqvist's beta and van
der Waerden's coefficient, which are denoted by $\rho_\text{S}$, $\beta$ and
$\zeta$, respectively. To this end, let $\mathcal S_d$, $\mathcal B_d$ and
$\mathcal W_d$ be the set of $d\times d$-matrices of Spearman's rho, Blomqvist's
beta and van der Waerden's coefficients, respectively.  As is done in the
previous subsection, denote by $\mathcal P_d$ the set of all
$d\times d$-correlation matrices, that is, the set of all symmetric, positive
semi-definite matrices in $[-1,1]^d$ with diagonal elements one.  It is
well-known that $\mathcal P_d$ is a convex set for any $d\ge 1$.  Let
$\mathcal P_d^{\text{U}}$ and $\mathcal P_d^{\text{B}}(p)$, $p \in (0,1)$, be
the set of all correlation matrices of $d$-dimensional random vectors whose
marginals are all $\U(0,1)$ and all $\Bern(p)$, respectively. By
Corollary~\ref{convexity of K_G set}, $\mathcal P_d^{\text{U}}$ and
$\mathcal P_d^{\text{B}}(p)$ are also convex sets. We can now characterize the
sets $\mathcal S_d$, $\mathcal B_d$ and $\mathcal W_d$.

\begin{proposition}[Characterizations of $\mathcal S_d$, $\mathcal B_d$ and $\mathcal W_d$]\label{prop:characterize:S:B:W}
  \begin{enumerate}
  \item\label{prop:characterize:S:B:W:1} $\mathcal P_d^{\text{U}} = \mathcal P_d$ for
    $d \leq 9$, that is, the set of correlation matrices of random vectors with
    standard uniform marginals coincides with the set of correlation matrices
    for $d\leq 9$.  For $d \geq 10$, $\mathcal P_d^{\text{U}} \subseteq \mathcal
    P_d$.
  \item\label{prop:characterize:S:B:W:2} $\mathcal S_d = {\mathcal P}_d^{\text{U}}$,
    that is, the set of Spearman's rho matrices coincides with the set of
    correlation matrices of random vectors with standard uniform marginals.
  \item\label{prop:characterize:S:B:W:3}
    $\mathcal B_{d} = \mathcal P_d^{\text{B}}(1/2)$, that is, the set of Blomqvist's
    beta matrices coincides with the set of correlation matrices of random
    vectors with symmetric Bernoulli marginals.
  \item\label{prop:characterize:S:B:W:4} $\mathcal W_d = \mathcal P_d$, that
    is, the set of van der Waerden's matrices coincides with the set
    of all correlation matrices.
  \end{enumerate}
\end{proposition}
\begin{proof}
  \ref{prop:characterize:S:B:W:1} is from \cite{devroye2015copulas}, and \ref{prop:characterize:S:B:W:2} and \ref{prop:characterize:S:B:W:4} are direct consequences of the definition of Spearman's rho and van der Waerden's coefficient. We thus have left to prove \ref{prop:characterize:S:B:W:3}.
  Consider ``$\subseteq$''. Let $(\beta_{ij}) \in \mathcal B_d$.
  Then there exists a $d$-dimensional random vector $\bm X$ such that $\beta(X_i,X_j)=\beta_{ij}$.
  By Example~\ref{ex:spearman:beta:waerden}~\ref{ex:spearman:beta:waerden:b},
  \begin{align*}
    \beta_{ij}=\rho(G^\i(F_i(X_i)),G^\i(F_j(X_j))),\quad i,j=1,\dots,d,
  \end{align*}
  where $G$ is the distribution function of $\Bern(1/2)$.
  Since $G^\i(F_i(X_i)),G^\i(F_j(X_j))\sim \Bern(1/2)$, we obtain that $(\beta_{ij}) \in \mathcal P_d^{\text{B}}(1/2)$.
  Now consider ``$\supseteq$''. Let $\bm{B}=(B_1,\dots,B_d)$ be a
  $d$-dimensional symmetric Bernoulli random vector with correlation matrix
  $\rho(\bm{B})=(\rho_{ij})$. Let $C$ be any copula such that
  \begin{align*}
    \P(B_1\leq b_1,\dots,B_d\leq b_d) = C(\P(B_1\leq b_1),\dots,\P(B_d\leq b_d)).
  \end{align*}
  Since, for $j=1,\dots,d$,
  \begin{align*}
    \P(B_j\leq b_j)=\begin{cases}
      0, & \text{if}\ b_j < 0, \\
      1/2, & \text{if}\ 0 \leq b_j <1,\\
      1, & \text{if}\ b_j\ge 1,\\
    \end{cases}
  \end{align*}
  $C$ is only uniquely determined in $(1/2,\dots,1/2)$ inside $[0,1]^d$.
  Furthermore, for any $(j_1,\dots,j_d)\in \{0,1\}^{d}$, the following identity holds:
  \begin{align*}
    C((1/2)^{j_1},\dots,(1/2)^{j_d}) = \P(B_1 \leq 1-j_1,\dots,B_d \leq 1-j_d).
  \end{align*}
  Let $\overline C$ be the survival copula of $C$ and $\bm{U} \sim \overline C$, so $\bm{1}-\bm{U}\sim C$; in particular,
  the marginals $F_1,\dots,F_d$ of $\bm{U}$ are $\U(0,1)$.
  Let $G(p)=\I_{\{p>1/2\}}$ be the distribution function of the symmetric Bernoulli distribution.
  Then
  \begin{align*}
    &\phantom{{}={}}\P(G^\i(U_1)\leq 1-j_1,\dots,G^\i(U_d)\leq 1-j_d)\\
    &= \P(\I_{\{U_1>1/2\}}\leq 1-j_1,\dots,\I_{\{U_d>1/2\}}\leq 1-j_d)\\
    &=\P(1-U_1 \leq (1/2)^{j_1},\dots,1-U_d \leq (1/2)^{j_d})
    =C((1/2)^{j_1},\dots,(1/2)^{j_d})\\
    &= \P(B_1 \leq 1-j_1,\dots,B_d \leq 1-j_d),\quad (j_1,\dots,j_d)\in \{0,1\}^{d}.
  \end{align*}
  Therefore, we have that $\bm{B}=(B_1,\dots,B_d) \deq (G^\i(U_1),\dots,G^\i(U_d))$.
  Consequently,
  \begin{align*}
    \beta(U_i,U_j) = \rho(G^\i(F_i(U_i)),G^\i(F_j(U_j))) =\rho(G^\i(U_i),G^\i(U_j)) = \rho(B_i,B_j) = \rho_{ij}.
  \end{align*}
  Since the random vector $\bm{U}$ attains $(\rho_{ij})$ as its Blomqvist's beta matrix, we have $(\rho_{ij}) \in \mathcal B_{d}$.
\end{proof}

Concerning
Proposition~\ref{prop:characterize:S:B:W}~\ref{prop:characterize:S:B:W:1},
\cite{devroye2015copulas} conjectured that the inclusion relationship among
$\mathcal P_d^{\text{U}}$ and $\mathcal P_d$ is strict for $d \geq 10$ .  Later
\cite{wang2018compatible} revealed that $\mathcal P_d$ is strictly larger than
$\mathcal P_d^{\text{U}}$ for $d\geq 12$. Although a complete characterization
of $\mathcal P_d^{\text{U}}$ is still unknown for $d \geq 10$, it is known that
$\mathcal P_d^{\text{U}}$ and $\mathcal P_d$ are not significantly different for
any $d\geq 1$ as explained in the following remark.

\begin{remark}[$\mathcal S_d$ and $\mathcal P_d$]
    Even for $d\geq 10$, $\mathcal S_d$ and $\mathcal P_d$ cannot be
    largely different since a Gauss copula with correlation parameter
    $P = (\rho_{ij})\in \mathcal P_{d}$ has Spearman's rho matrix
    $(\rho_{\text{S},ij})$ with
    $\rho_{\text{S},ij} = (6/\pi)\arcsin(\rho_{ij}/2)$, or equivalently,
    $\rho_{ij}=2\sin(\pi \rho_{\text{S},ij}/6)$. Since
    $|\rho_{\text{S},ij}-\rho_{ij}|=|\rho_{\text{S},ij}-2\sin(\pi
    \rho_{\text{S},ij}/6)|\leq 0.0181$, one can find an elementwise close
    Spearman's rho matrix attained by a Gauss copula for every correlation
    matrix $P\in \mathcal P_{d}$.
\end{remark}

The consequences of Proposition~\ref{prop:characterize:S:B:W} related to the
compatibility problem are as follows.  First,
Proposition~\ref{prop:characterize:S:B:W}~\ref{prop:characterize:S:B:W:1} and
\ref{prop:characterize:S:B:W:2} allow one to check that a given
$d\times d$-matrix for $d\le 9$ is $\rho_{\text{S}}$-compatible via checking
whether the matrix is a correlation matrix, for example, by trying to compute
its Cholesky factor.  For $d\geq 10$, no straightforward way to check
$\rho_{\text{S}}$-compatibility is available yet while the sufficient condition
in Proposition~\ref{prop:suff:cond:k:compatibility} is still valid.  Second,
Proposition~\ref{prop:characterize:S:B:W}~\ref{prop:characterize:S:B:W:3} states
that the set of all Blomqvist's beta matrices are completely characterized by
the set of correlation matrices of random vectors with symmetric Bernoulli
margins.  In Subsection~\ref{subsection: Bernoulli(1/2)-compatibility problem},
we will discuss the problem to check a given matrix belongs to
$\mathcal P_d^{\text{B}}(1/2)$.  Finally,
Proposition~\ref{prop:characterize:S:B:W}~\ref{prop:characterize:S:B:W:4} says
that the set of van der Waerden's matrices coincides with the set of all
correlation matrices, and thus, checking $\zeta$-compatibility is
straightforward. In terms of checking compatibility, this property of van der
Waerden's coefficient is an attractive feature that $\rho_S$ and $\beta$ do not
satisfy for any dimension $d\geq 1$. Note that this property is not unique to
van der Waerden's coefficient but holds for any elliptical distribution $G$ with
finite second moments; see \cite[Chapter~4]{joe1997}.

\subsection{$\Bern(1/2)$-compatibility problem}\label{subsection: Bernoulli(1/2)-compatibility problem}
As we have seen in Section~\ref{sec:cor:based:measures} and \ref{sec:compat} so
far, $\mathcal P_{d}(1/2)$ plays important roles when studying matrix
compatibility problems since it coincides with $\mathcal B_d$, the set of all
Blomqvist's beta matrices, and $\mathcal P_{d}(1/2) \subseteq \mathcal{K}_{G}$,
the set of all $\kappa_{G}$-compatible matrices.  If
$P \in \mathcal P_{d}(1/2)$, we call $P$ \emph{$\Bern(1/2)$-compatible}.  In
this section, we address the \emph{membership testing problem} for
$\mathcal P_{d}(1/2)$, that is, a test whether a given matrix is
$\Bern(1/2)$-compatible or not.

\cite{huber2017bernoulli} presented a characterization of the set
$\mathcal P_{d}^{\text{B}}(1/2)$ which can be used for membership testing as we now
explain.  For $l=1,\dots,2^{d-1}$, let $\bbb(l) = (b_1,\dots,b_d)$ be the binary
expansion of $l$, that is,
\begin{align*}
  \bbb(l) = (b_1,\dots,b_d)\quad \text{if and only if} \quad l = 1 + \sum_{j=1}^{d}b_{j}2^{d-j}.
\end{align*}
Note that $b_1$ is equal to 0 for all $l=1,\dots,2^{d-1}$.
For each $l$, let $\pi_l$ be the $d$-dimensional distribution which puts equal mass on $\bbb(l) = (b_1,\dots,b_d)$ and $\bm{1}-\bbb(l) = (1-b_1,\dots,1-b_d)$.
One can easily check %
that the correlation matrix of $\bm{X} \sim \pi_l$ is given by
\begin{align*}
  \rho(X_i,X_j) = 2\I_{\{b_i(l) = b_j(l)\}} -1,\quad i,j = 1,\dots,d,
\end{align*}
where $b_i(l)$ denotes the $i$th element of $\bm{b}(l)$.
This leads to the following characterization of the set $\mathcal P_d^{\text{B}}(1/2)$; see \cite{huber2017bernoulli}.

\begin{theorem}[Characterization of $\mathcal P_d^{\text{B}}(1/2)$]\label{theorem: characterization of bernoulli(1/2) correlations}
  $\mathcal P_d^{\text{B}}(1/2)$ is the convex hull of correlation matrices of the two-point distributions $\pi_{1},\dots,\pi_{2^{d-1}}$, that is,
  \begin{align*}
    \mathcal P_d^{\text{B}}(1/2) &= \conv\{\rho(\pi_l):\ l=1,\dots,2^{d-1}\}\\
                          &= \biggl\{\,\sum_{l=1}^{2^{d-1}}\alpha_l \rho(\pi_l):\  \alpha_1,\dots,\alpha_{2^{d-1}}\geq 0, \ \alpha_1+ \cdots +  \alpha_{2^{d-1}} = 1 \biggr\},
  \end{align*}
  where $\rho(\pi_l)$ is the correlation matrix of $\pi_l$.
\end{theorem}

\begin{remark}[Cut polytope and elliptope]
  By Theorem~\ref{theorem: characterization of bernoulli(1/2) correlations},
  $\mathcal P_d^{\text{B}}(1/2)$ coincides with the set known as a \emph{cut
    polytope}, which is the collection of matrices $\bc\bc\T$ for all
  $\bc \in \{-1,1\}^d$.  Moreover, its positive semi-definite relaxation is
  known to be the \emph{elliptope} $\mathcal P_d$; see
  \cite{laurent1995positive} and \cite{tropp2018simplicial}.
\end{remark}

\begin{example}[Cases $d=2$ and $d=3$]
  Write $P=(\rho_{ij})\in \mathcal P_d^{\text{B}}(1/2)$.  When $d=2$,
  $\rho_{12}=\rho_{21}$ and $\rho_{12}$ can take any value from $-1$ to $1$
  since $\rho_{12}=\alpha (+1) + (1-\alpha) (-1) = 2\alpha - 1$ for
  $\alpha \in [0,1]$.  When $d=3$, the characterization in Theorem~\ref{theorem:
    characterization of bernoulli(1/2) correlations} reduces to
  \begin{align}\label{tetrahedron}
    -1 \leq \sum_{1\le i<j \le 3}\rho_{ij} \leq 1 + 2 \min_{1\le i,j\le 3}\{\rho_{ij}\}.
  \end{align}
  In terms of the triple $(\rho_{12},\rho_{13},\rho_{23})$ of correlations, \eqref{tetrahedron} forms a tetrahedron with vertices $(1,1,1)$, $(1,0,0)$, $(0,1,0)$ and $(0,0,1)$.
  One can check that $\mathcal P_{d}^{\text{B}}(1/2)$ is a strict subset of $\mathcal P_{d}$ for $d\geq 3$.
  For instance, consider a matrix of the form
  \begin{align*}
    P(\rho)=\begin{pmatrix}
        1&  \rho & \rho  \\
        \rho & 1 & \rho  \\
        \rho & \rho &  1 \\
      \end{pmatrix}.
  \end{align*}
  $P(\rho)$ is a proper correlation matrix if and only if
  $-1/2\leq \rho \leq 1$.  On the other hand, the inequality in
  \eqref{tetrahedron} says that $P(\rho) \in \mathcal P^{\text{B}}_d(1/2)$ if and only
  if $-1/3 \leq \rho \leq 1$.  Therefore, if $-1/2\leq \rho < -1/3$, then
  $P(\rho)$ belongs to $\mathcal P_d$ but not to $\mathcal P_{3}^{\text{B}}(1/2)$.
\end{example}

The characterization in Theorem~\ref{theorem: characterization of bernoulli(1/2)
  correlations} provides a method to check that a given matrix is
$\Bern(1/2)$-compatible.
\begin{proposition}[Checking $\Bern(1/2)$-compatibility]
  A given matrix $P=(\rho_{ij})$ is $\Bern(1/2)$-compatible if and only if there
  exist $\alpha_1,\dots,\alpha_{2^{d-1}}\geq 0$ such that the following
  $1+ d(d-1)/2$ equations hold:{}
  \begin{align*}
    \alpha_1 + \cdots+\alpha_{2^{d-1}} = 1, \quad \sum_{l=1}^{2^{d-1}}\alpha_l \I_{\{b_i(l) = b_j(l)\}} = \frac{\rho_{ij}+1}{2},\quad 1\leq i<j\leq d.
  \end{align*}
  Equivalently, the following \emph{phase~I linear program} attains zero:
  \begin{align}
    \label{phase I LP} \min z_1 + \cdots + z_{2^{d-1}}\quad\text{subject to}\ \begin{cases}
      D{\bm \alpha} + {\bm z} = {\bm \lambda},\\
      {\bm \alpha},{\bm z} \geq {\bm 0},
    \end{cases}
  \end{align}
  where ${\bm \alpha} = (\alpha_1,\dots,\alpha_{2^{d-1}})\in [0,1]^{2^{d-1}}$, $\bm \lambda =(\lambda_{12},\lambda_{13},\lambda_{23},\dots,\lambda_{d-1\,d},1)\in [0,1]^{1+d(d-1)/2}$ for $\lambda_{ij}=(\rho_{ij}+1)/2$ and
  \begin{align*}
    D =
    \begin{pmatrix}
      \I_{\{b_1(1) = b_2(1)\}}  & \I_{\{b_1(2) = b_2(2)\}} &\cdots & \I_{\{b_{1}(2^{d-1}) = b_{2}(2^{d-1})\}} \\
      \I_{\{b_1(1) = b_3(1)\}}  & \I_{\{b_1(2) = b_3(2)\}} &\cdots & \I_{\{b_{1}(2^{d-1}) = b_{3}(2^{d-1})\}} \\
      \I_{\{b_2(1) = b_3(1)\}}  & \I_{\{b_2(2) = b_3(2)\}} &\cdots & \I_{\{b_{2}(2^{d-1}) = b_{3}(2^{d-1})\}} \\
      \vdots                  & \vdots                     &\vdots &  \vdots \\
      \I_{\{b_{d-1}(1) = b_{d}(1)\}} &  \I_{\{b_{d-1}(2) = b_{d} (2)\}}  &\cdots &  \I_{\{b_{d-1}(2^{d-1}) = b_{d}(2^{d-1})\}} \\
      1                  & 1                     &\cdots &    1\\
    \end{pmatrix} \in \{0,1\}^{\bigl(1+\frac{d(d-1)}{2}\bigr)\times 2^{d-1}}\!\!\!\!.
  \end{align*}
\end{proposition}
Note that the set of constraints in \eqref{phase I LP} is always nonempty since
$(\bm \alpha,\bm z) = (\bm 0,\bm \lambda)$ is a feasible solution.  The phase~I
linear program can be solved, for example, with the \R\ package \textsf{lpSolve}
although it is computationally demanding for large $d$. This is to be expected
since such problems are known to be NP-complete; see \cite{pitowsky1991correlation}.

Once a (componentwise) non-negative vector $\bm \alpha^{\ast}$ such that
$D\bm \alpha^{\ast} = \bm \lambda$ is obtained, the corresponding symmetric
Bernoulli random vector $\bm{B}$ with correlation matrix $P=(\rho_{ij})$ can be
simulated by the following algorithm, which enables us to solve the
attainability problem discussed in Section~\ref{subsec: Attainability}.

\begin{algorithm}[Simulating random vectors with $\Bern(1/2)$ marginals and given correlation matrix $P$]
\label{algorithm: generate symmetric bernoulli}
  \begin{enumerate}
  \item For $P$, solve \eqref{phase I LP} to find $(\alpha_{1},\dots,\alpha_{2^{d-1}})$.
  \item Choose the index $l$ with probability $\alpha_{l}$, $l\in \{1,\dots,2^{d-1}\}$.
  \item Set $\bm{B} = \bbb(l)$ or $\bm{1} - \bbb(l)$ with probability $1/2$ each.
  \end{enumerate}
\end{algorithm}

\begin{example}[Numerical example for $d=3$]
Consider the two $3\times 3$ matrices
  \begin{align*}
  P_1=\begin{pmatrix}
  1 & -0.95 & 0.5 \\
  -0.95 & 1 & -0.4 \\
  0.5 & -0.4 & 1 \\
        \end{pmatrix},\quad
  P_2=\begin{pmatrix}
  1 & -0.9 & 0.5 \\
  -0.9 & 1 & -0.4 \\
  0.5 & -0.4 & 1 \\
\end{pmatrix},
  \end{align*}
  both of which can be shown to be positive definite, so correlation matrices.
  For $d=3$, the numbers $l=1,\dots,2^{d-1}=4$ have the binary expansions $\bbb(1)=(0,0,0)$,
  $\bbb(2)=(0,0,1)$, $\bbb(3)=(0,1,0)$ and $\bbb(4)=(0,1,1)$.
  The corresponding matrix $D$ is then given by
  \begin{align*}%
    D = \begin{pmatrix}
        1 & 1 & 0 & 0\\
        1 & 0 & 1 & 0\\
        1 & 0 & 0 & 1\\
        1 & 1 & 1 & 1\\
      \end{pmatrix}.
  \end{align*}
  For $P_1$,
  $\lambda_1=(\lambda_{1,12},\lambda_{1,13},\lambda_{1,23},1)=(0.025, 0.750,
  0.300, 1.000)$.  Solving the phase I linear program with the \R\ package
  \textsf{lpSolve} yields the minimum 0.025 of the objective function
  $z_1+z_2+z_3+z_4$, which does not attain zero. Therefore, although
  $P_1$ is a proper correlation matrix, it is not $\Bern(1/2)$-compatible.  For
  $P_2$, $\lambda_2 = (0.050, 0.750, 0.300, 1.000)$.  By using \textsf{lpSolve},
  the objective function is found to achieve zero, and we thus numerically
  checked that $P_2 \in \mathcal P_d^{\text{B}}(1/2)$. These results can also be
  confirmed with the inequality in \eqref{tetrahedron}.
\end{example}

One can thus check the compatibility of Blomqvist's beta matrices (or,
equivalently, correlation matrices of random vectors with symmetric Bernoulli
margins) by solving the phase I linear program \eqref{phase I LP} and checking
whether the objective function attains zero. By the same procedure, the
sufficient condition shown in Proposition~\ref{prop:suff:cond:k:compatibility} can
also be checked for general $\kappa_G$ compatibility.

\subsection{Attainability of matrices of measures of concordance}\label{subsec: Attainability}
We now consider the attainability problem. We call a $\kappa_{G}$-compatible
matrix $P\in[-1,1]^{d\times d}$ \emph{$\kappa_{G}$-attainable} if one can
construct a random vector $\bm{X}=(X_1,\dots,X_d)$
such that $\kappa_{G}(\bm{X})=P$.  The proof of
Proposition~\ref{prop:suff:cond:k:compatibility} already indicates such a
construction principle for a $d$-dimensional random vector $\bm{X}$ such that,
for a given matrix $P\in \mathcal P_{d}^{\text{B}}(1/2)$, one has
$\kappa_{G}(\bm{X})=P$.
\begin{corollary}[$\kappa_{G}$-attainability of $P\in\mathcal{B}_d = \mathcal P_{d}^{\text{B}}(1/2)$]\label{corollary: kappa attainability}
  Let $P\in \mathcal P_{d}^{\text{B}}(1/2)$ and the representation $P=\sum_{l=1}^{2^{d-1}}\alpha_{l}\rho(\pi_l)$
  according to Theorem~\ref{theorem: characterization of bernoulli(1/2) correlations} be given.
  Then $P$ is $\kappa_{G}$-attainable by $\bm{X}=(X_1,\dots,X_d)$ defined by
  \begin{align}
    X_j = B_jU+ (1-B_j)(1-U),\quad j=1,\dots,d,\label{random vector attains kappa}
  \end{align}
  where $U\sim \U(0,1)$ and ${\bm B}=(B_1,\dots,B_d)$ is constructed as in Algorithm~\ref{algorithm: generate symmetric bernoulli}.
\end{corollary}
Since $\mathcal B_d=\mathcal P_{d}^{\text{B}}(1/2)$, that is, the set of Blomqvist's
beta matrices coincide with the set of correlations of random vectors with
symmetric Bernoulli marginals, all matrices $P\in \mathcal B_d$ can be
attained by \eqref{random vector attains kappa}.

Next, for matrices of pairwise van der Waerden's coefficients $\zeta$,
any $\zeta$-compatible matrix is attainable by multivariate normal distribution.

\begin{corollary}[$\zeta$-attainability of $P\in\mathcal{W}_d=\mathcal{P}_d$]
\label{corollary: attainability of van der Waerden}
  Any matrix $P \in \mathcal W_d$ is attainable by the multivariate normal
  distribution with covariance matrix $P$.
\end{corollary}

Finally, for Spearman's rho, $\rho_{\text{S}}$-attainability is not completely solved for dimensions $d\geq 3$.
If $P\in \mathcal P_{d}^{\text{B}}(1/2)$, $P$ is $\rho_{\text{S}}$-attainable by Corollary \ref{corollary: kappa attainability} for $d\geq 3$.
If $P\notin \mathcal P_{d}^{\text{B}}(1/2)$, $P$ is known to be $\rho_{\text{S}}$-attainable only when $d=3$ by the results in \cite{hurlimann2012trivariate}, \cite{hurlimann2014closed} and \cite{kurowicka2001conditional}, where \emph{universal copulas} are studied, that is,
explicitly constructed copulas with given correlation matrices.
For $d \geq 4$, such a universal copula is still unknown to the best of our knowledge.
Accordingly, a general $\rho_{\text{S}}$-compatible matrix $P$ is not known to be attainable when $d\geq 4$.

\section{Compatibility and attainability for block matrices}\label{sec:block:mat}
In this section, we study the compatibility and attainability of \emph{block
  matrices} $P$, that is, matrices containing homogeneous blocks (so blocks of
equal entries), possibly with ones on the diagonal. A special case of block
matrices are hierarchical matrices, which are introduced in Example~\ref{ex:hier:matrices}.
Block matrices naturally appear when clustering algorithms are applied to
matrices of measures of concordance or when (rather) sparse, partially
exchangeable hierarchical models are designed.

Although all the criteria introduced in Section~\ref{sec:compat} can be
directly applied to block correlation matrices, the corresponding computational
effort can be large, especially when $d$ is large. The comparably small number of
different entries in block or hierarchical matrices is especially attractive for
high-dimensional modeling and one expects more efficient ways to check
compatibility and attainability for such matrices.  Specifically, compatibility
and attainability for Spearman's rho matrices are in demand since, as discussed
in Section~\ref{sec:compat}, there is no method available to check
compatibility for $d\geq 10$, and to check attainability for $d\geq 4$.

\subsection{Definition and notations}
We consider the following symmetric matrix in $[-1,1]^{d\times d}$ with diagonal entries equal to one:
\begin{align}
  P =\begin{pmatrix}
    P_{11} & \cdots & P_{1S} \\
    \vdots & \ddots & \vdots \\
    P_{S1} & \cdots & P_{SS} \\
  \end{pmatrix},\quad\text{for}\ P_{s_1s_2}=
  \begin{cases}
      (1-\rho_{ss})I_{d_s} + \rho_{ss}J_{d_s},&\text{if}\ s_1 = s_2 = s,\\
      \rho_{s_1s_2}J_{d_{s_1}d_{s_2}},&\text{if}\ s_1\neq s_2,
    \end{cases}
  \label{block correlation structure}
\end{align}
where $I_{d_s}$ denotes the $d_s\times d_s$ identity matrix,
$J_{d_{s_1}d_{s_2}}=\bm{1}_{d_{s_1}}\bm{1}_{d_{s_2}}\T\in\IR^{d_{s_1}\times
  d_{s_2}}$ (for $\bm{1}_{d_s}=(1,\dots,1)\in\IR^{d_s}$) is the
$d_{s_1}\times d_{s_2}$ matrix of ones and $J_{d_s}=J_{d_sd_s}$.  We call a
matrix of the form~\eqref{block correlation structure} a \emph{block homogeneous
  matrix}.  For notational convenience, let
\begin{align*}
  \Gamma_d(a,b) = a I_{d} + b (J_d- I_{d})=(a-b)I_{d} + bJ_d
\end{align*}
which is also known as the $d$-dimensional \emph{compound symmetry matrix}.
With this notation, the matrices on the diagonal of $P$ in \eqref{block correlation structure}
can be written as $P_{ss}=\Gamma_{d_s}(1,\rho_{ss})$.

A matrix of the form~\eqref{block correlation structure} appears, for
example, as a correlation matrix of a random vector with homogeneous
correlations within blocks.  Let $\bm X = (X_1,\dots,X_d)$ be a $d$-dimensional
random vector which can be divided into $S$ such blocks or groups
\begin{align*}
  \bm{X}=(\bm{X}_1,\dots,\bm{X}_S)=(X_{11},\dots,X_{1d_1},\dots,X_{S1},\dots,X_{Sd_S}),
\end{align*}
where $d_s$ is the size of group $s\in\{1,\dots,S\}$.
In financial and insurance applications, the groups are often industry sectors, business sectors, regions, etc.
If we consider the case where the correlation between two random variables
depends only on the groups they belong to, then the resulting correlation matrix
of $\bm X$ is block homogeneous of the form~\eqref{block correlation structure}
where $\rho_{s_1s_2}$ represents the correlation coefficient within two
(possibly equal) groups $s_1$ and $s_2$.

When we call a matrix $P$ block homogeneous, it is a symmetric
$[-1,1]^{d\times d}$ matrix with diagonal entries equal to one, but not
necessarily a correlation matrix since positive definiteness of $P$ is not
assumed. Note that, for compound symmetry matrices, it is well-known that
$\Gamma_d(a,b)$ is positive definite if and only if $-a/(d-1) < b < a$.
Therefore, $P_{ss}$, $s=1,\dots,S$, is positive definite if and only if
$ -1/(d_s -1) < \rho_{ss} < 1$.
\begin{example}[Hierarchical matrices]\label{ex:hier:matrices}
  Consider the block homogeneous matrix
  \begin{align}
    P= \left(\ \ \begin{matrix}
        \cc{0.5} 1 & \cc{0.5} 0.4 & \cc{0.5} 0.4 & \cc{0.5} 0.4 & 0.1 & 0.1 & 0.1 & 0.1 & 0.1\\
        \cc{0.5} 0.4 & \cc{0.5} 1 & \cc{0.5} 0.4 & \cc{0.5} 0.4 & 0.1 & 0.1 & 0.1 & 0.1 & 0.1\\
        \cc{0.5} 0.4 & \cc{0.5} 0.4 & \cc{0.5} 1 & \cc{0.5} 0.4 & 0.1 & 0.1 & 0.1 & 0.1 & 0.1\\
        \cc{0.5} 0.4 & \cc{0.5} 0.4 & \cc{0.5} 0.4 & \cc{0.5} 1 & 0.1 & 0.1 & 0.1 & 0.1 & 0.1\\
        0.1 & 0.1 & 0.1 & 0.1 & \cc{0.65} 1 & \cc{0.65} 0.3 & \cc{0.65} 0.3 & 0.15 & 0.15\\
        0.1 & 0.1 & 0.1 & 0.1 & \cc{0.65} 0.3 & \cc{0.65} 1 & \cc{0.65} 0.3 & 0.15 & 0.15\\
        0.1 & 0.1 & 0.1 & 0.1 & \cc{0.65} 0.3 & \cc{0.65} 0.3 & \cc{0.65} 1 & 0.15 & 0.15\\
        0.1 & 0.1 & 0.1 & 0.1 & 0.15 & 0.15& 0.15 & \cc{0.8} 1 & \cc{0.8} 0.2\\
        0.1 & 0.1 & 0.1 & 0.1 & 0.15 & 0.15& 0.15 & \cc{0.8} 0.2 & \cc{0.8} 1\\
      \end{matrix}\ \ \right)\label{eq:hier:mat}
  \end{align}
  with $S=3$, $(d_1,d_2,d_3)=(4,3,2)$, $(\rho_{11},\rho_{22},\rho_{33},\rho_{12},\rho_{13},\rho_{23})=(0.4,0.3,0.2,0.1,0.1,0.15)$.
  This matrix can be described by a tree $T_P$ illustrated in Figure~\ref{fig:tree:hier:mat}.
  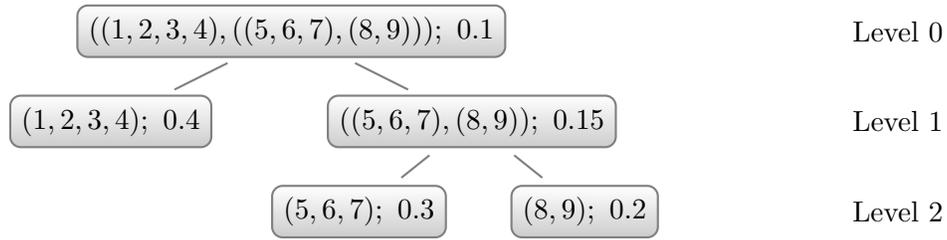
\begin{figure}[htbp]
    \centering
    \begin{tikzpicture}[
      grow = south,
      level 1/.style = {sibling distance = 48mm, level distance = 12mm},
      level 2/.style = {sibling distance = 30mm, level distance = 12mm},
      edge from parent/.style = {thick, draw = black!50},
      mygeneralstyle/.style = {rectangle, rounded corners, shade, top color = white,
        bottom color = black!20, draw = black!50, thick, inner sep = 1.5mm},
      mynodestyle/.style = {mygeneralstyle, minimum width = 12mm, minimum height = 6mm},
      myleafstyle/.style = {mygeneralstyle, minimum width =  7mm, minimum height = 5mm}
      ]
      \node[mynodestyle](nodelevel0){$((1,2,3,4),((5,6,7),(8,9)));\ 0.1$}
      child[edge from parent/.style = {thick, draw = black!50, shorten >= 1.5mm, shorten <= 1.5mm}]{
        node[mynodestyle]{$(1,2,3,4);\ 0.4$}{
        }
      }
      child[edge from parent/.style = {thick, draw = black!50, shorten >= 1.5mm, shorten <= 1.5mm}]{
        node[mynodestyle](nodelevel1){$((5,6,7),(8,9));\ 0.15$}{
          child[edge from parent/.style = {thick, draw = black!50, shorten >= 1.5mm, shorten <= 1.5mm}]{
            node[mynodestyle]{$(5,6,7);\ 0.3$}{
            }
          }
          child[edge from parent/.style = {thick, draw = black!50, shorten >= 1.5mm, shorten <= 1.5mm}]{
            node[mynodestyle](nodelevel2){$(8,9);\ 0.2$}{
            }
          }
        }
      };
      \node[right=44.7mm of nodelevel0]{Level~0};
      \node[right=30mm of nodelevel1]{Level~1};
      \node[right=24.4mm of nodelevel2]{Level~2};
    \end{tikzpicture}
    \caption{Tree representation $T_P$ of the hierarchical correlation matrix $P$ in \eqref{eq:hier:mat}.}
    \label{fig:tree:hier:mat}
  \end{figure}
  For the tree $T_P$, denote by $v_{lm}$ the $m$th node (counted from left) at
  level $l\in \{0,1,2\}$.  The \emph{leaves} (that is, the terminal nodes)
  $v_{11},v_{21}$ and $v_{22}$ represent the groups of variable indices
  $(1,2,3,4)$, $(5,6,7)$ and $(8,9)$, respectively.  Nodes $v_{21}$ and $v_{22}$
  are connected by a node $v_{12}$, and $v_{11}$ and $v_{12}$ are connected by a
  node $v_{01}$.  To each vertex $v=v_{01},v_{11},v_{12},v_{21},v_{22}$ (in the
  set of vertices denoted by
  $\mathcal V = \{v_{01},v_{11},v_{12},v_{21},v_{22}\}$), a single number
  $\rho_{v}=0.4, 0.3, 0.2, 0.15, 0.1$ is attached, respectively. The vertex
  $v_{01}$ at the lowest level is called \emph{root}; if two nodes $v$ and $v'$
  are connected and $v$ is at lower level than $v'$, then $v$ is a
  \emph{parent} of $v'$ and $v'$ is a \emph{child} of $v$. A node $v$ is called
  \emph{descendant} of another node $v'$ if $v$ is in the shortest path from
  $v'$ to the root of the tree; note that each single node is regarded as a
  descendant of itself.  Finally, for a pair of two nodes $(v,v')$, the
  \emph{lowest common ancestor} is the lowest node that has both $v$ and $v'$ as
  descendants; when $v=v'$, the lowest common ancestor is $v$
  itself.

  With these notions, the block matrix $P$ is recovered from the tree $T_P$ by
  defining a matrix with diagonal entries equal to $1$ and the $(i,j)$-entry,
  for $i\neq j$, equal to the number attached to the descendant of $(v_i, v_j)$
  where $v_i$ and $v_j$ are the leaves of groups of variable indices containing
  $i$ and $j$, respectively. If a block homogeneous correlation matrix $P$
  admits such a tree representation $T_P$, we call $P$ \emph{hierarchical
    matrix} and $T_P$ the corresponding \emph{hierarchical tree}. The
  matrix~\eqref{eq:hier:mat} is thus a hierarchical matrix with corresponding
  tree given in Figure~\ref{fig:tree:hier:mat}.
\end{example}

\subsection{Positive (semi-)definiteness}
By Corollary~\ref{corollary:upper:lower:bounds}, positive (semi-)definiteness is a necessary condition for compatibility
of %
matrices of transformed rank correlation coefficients including Spearman's rho,
Blomqvist's beta and van der Wearden's coefficient. In the case of van der
Wearden's coefficient, it is even sufficient for
compatibility. If a matrix is block homogeneous,
it turns out to suffice to check positive semi-definiteness of an $S\times S$ matrix,
see Theorem~\ref{theorem: iff condition for positive semi-definiteness of block correlation matrix} below.
This result can lead to a significant reduction in the computational effort.

\begin{definition}[Block average map]
Let $P$ be a block homogeneous matrix of form \eqref{block correlation structure}.
The \emph{block average map} $P\mapsto \phi(P)$ for $P=(\rho_{ij})\in\IR^{d\times d}$ is defined by
\begin{align*}
  \phi(P)=\begin{pmatrix}
    \tilde \rho_{11} & \rho_{12} & \cdots & \rho_{1S}\\
    \rho_{21} & \ddots & \ddots & \vdots\\
    \vdots &  \ddots & \ddots  & \rho_{S-1\,S}\\
    \rho_{S1} & \cdots & \rho_{S-1\,S} & \tilde \rho_{SS}
    \end{pmatrix}\in\IR^{S\times S},\quad
    \tilde \rho_{ss} = \frac{1+(d_s - 1)\rho_{ss}}{d_s},\ s=1,\dots,S.
\end{align*}
\end{definition}

The block average map $\phi$ allows one to
collapse block matrices (to ``ordinary'' matrices).
If $\bm{X} = (X_1,\dots,X_d)$ is a
random vector with $\E(\bm{X}) = \bm 0$ and $\Cov(\bm{X})=P$ where $P$ is as in
\eqref{block correlation structure}, then $\bm{Y} = (\overline Y_1,\dots,\overline Y_S)$
defined by the group averages $\overline Y_s = \frac{1}{d_s}\sum_{j=1}^{d_s}X_{sj}$ has
covariance matrix $\phi(P)$, that is, $\Cov(\bm{Y}) = \phi(P)$.
\cite{roustant2017validity,huang2010correlation} showed that it suffices
to check positive (semi-)definiteness of the matrix $\phi(P)\in\IR^{S\times S}$
to obtain positive (semi-)definiteness of $P\in\IR^{d\times d}$.

\begin{theorem}[Characterization of positive (semi-)definiteness of block matrices]\label{theorem: iff condition for positive semi-definiteness of block correlation matrix}
  Let $P\in \IR^{d\times d}$ be a block matrix as in \eqref{block correlation
    structure}. Then $P$ is positive (semi-)definite if and only if $\phi(P)$ is
  positive (semi-)definite.
\end{theorem}
\begin{proof}
  See \cite{huang2010correlation} and \cite{roustant2017validity}.
\end{proof}

\begin{example}[Positive definiteness of a hierarchical matrix]\label{example: Positive definiteness of a hierarchical matrix}
  Consider $P$ as in \eqref{eq:hier:mat}, so $S=3$, $d=9$, $(d_1,d_2,d_3)=(4,3,2)$ with block average map given by
  \begin{align*}
    \phi(P)&=\begin{pmatrix}
        (1+(d_1-1)\rho_{11})/d_1 & \rho_{12} & \rho_{13}\\
        \rho_{21} & (1+(d_2-1)\rho_{22})/d_2 & \rho_{23}\\
        \rho_{31} & \rho_{32} & (1+(d_3-1)\rho_{33})/d_3\\
      \end{pmatrix}\\
           &=\begin{pmatrix}
             \frac{1+(4-1) 0.4}{4} & 0.1 & 0.1\\
             0.1 & \frac{1+(3-1)0.3}{3} & 0.15\\
             0.1 & 0.15 & \frac{1+(2-1)0.2}{2}\\
           \end{pmatrix}
    =\begin{pmatrix}
      0.55 & 0.1 & 0.1\\
      0.1 & 0.5\overline{3} & 0.15\\
      0.1 & 0.15 & 0.6\\
    \end{pmatrix}.
  \end{align*}
    One can easily check that $\phi(P)$ is positive definite.
    By Theorem~\ref{theorem: iff condition for positive semi-definiteness of block
    correlation matrix}, $P$ is thus positive definite.
\end{example}

\subsection{Block Cholesky decomposition}
The \emph{Cholesky decomposition} of a positive definite (positive
semi-definite) matrix $P \in \mathcal P_d$ is $P=LL\T$ for a lower triangular
matrix $L$ with positive (non-negative) diagonal elements, which is called the
\emph{Cholesky factor} of $P$. Such a decomposition of $P$ exists if and only if
$P$ is positive (semi-)definite and so can be used to check the latter property
computationally.

Cholesky decompositions are of utmost importance in various areas of statistics.
In quantitative risk management, they are frequently utilized to construct
multivariate elliptical distributions. For example, once the Cholesky factor $L$
of $P$ is computed, the $d$-dimensional random vector $\bm{X}=L\bZ$ satisfies
$\Cov(\bm{X})=LL\T=P$ for $\bZ\sim \N_d(0,I_d)$.  This $\bm X$ thus attains a
given matrix $P$ of van der Waerden's coefficients; see
Corollary~\ref{corollary: attainability of van der Waerden}. For building
hierarchical dependence models after estimating groups of homogeneous models or
after applying clustering algorithms (which naturally lead to groups of
variables), one often considers block homogeneous correlation matrices or hierarchical matrices (see Example~\ref{ex:hier:matrices}).  We will now
turn to the question how Cholesky factors of such matrices look like and can be
computed more efficiently than in the classical way.

\begin{proposition}[Cholesky factor of block matrices]\label{proposition: Cholesky factor of block matrices}
For a $d\times d$ block homogeneous correlation matrix $P$ of form \eqref{block correlation structure}, its Cholesky factor $L$ is of the form
\begin{align*}
  L = \begin{pmatrix}
    L_{11} & O & \cdots & O \\
    L_{21} & L_{22} & \ddots & O \\
    \vdots & \vdots & \ddots & \vdots\\
    L_{S1} & L_{S2} & \cdots & L_{SS} \\
  \end{pmatrix}
\end{align*}
where $O=(0)$ represents a block of zeros and, for $s=1,\dots,S$, the diagonal matrices are
\begin{align*}
  L_{ss}=\begin{pmatrix}
    \tilde  l_{ss,1} & 0 & 0 & \cdots & 0\\
    l_{ss,1} & \tilde  l_{ss,2} &0 &   & \vdots\\
    l_{ss,1} & l_{ss,2} & \tilde l_{ss,3}& \ddots & \vdots\\
    \vdots & \vdots & \vdots & \ddots &0\\
    l_{ss,1} & l_{ss,2} & l_{ss,3} & \cdots & \tilde l_{ss,d_s}\\
  \end{pmatrix}\in\IR^{d_s\times d_s}
\end{align*}
for some $\tilde  l_{ss,k},\ k=1,\dots,S$ and $ l_{ss,k},\ k=1,\dots,S-1$, and the off-diagonal matrices are
\begin{align*}
  L_{s+m, s} = (c_{sm,1}\bm{1}_{d_{s+m}},\dots,c_{sm,d_s}\bm{1}_{d_{s+m}})\in\IR^{d_{s+m}\times d_s},\quad m=1,\dots,S-s
\end{align*}
for some $(c_{sm,1},\dots,c_{sm,d_{s}})$.
\end{proposition}

The following algorithm computes the Cholesky factors of a given block
homogeneous correlation matrix; its proof thus shows Proposition~\ref{proposition: Cholesky
  factor of block matrices}.
\begin{algorithm}[Cholesky decomposition for block matrices]
  \label{algorithm: cholesky decomposition for block matrices}
  \begin{enumerate}
  \item Set $\overline P(1)=P$.
  \item For $s=1,\dots,S$ and $\overline P(s)$ of the form
    \begin{align}\label{block correlation matrix at iteration s}
      {\overline P}(s) =  \begin{pmatrix}
        P_{1,1}^{(s)} & \cdots & P_{1, S-s+1}^{(s)} \\
        \vdots & \ddots & \vdots \\
        P_{S-s+1, 1}^{(s)} & \cdots & P_{S-s+1,S-s+1}^{(s)}\\
      \end{pmatrix},
    \end{align}
    where, for $s_1,s_2 \in \{1,\dots,S-s+1\}$,
    \begin{align*}
      P_{s_1,s_2}^{(s)}=
      \begin{cases}
        \Gamma_{d_{s+t-1}}(\rho_{t}^{(s)},\rho_{t,o}^{(s)}),&\text{if}\ s_1 = s_2=t \in \{1,\dots,S-s+1\},\\
        \rho_{s_1,s_2}^{(s)}J_{d_{s+s_1-1}d_{s+s_2-1}},&\text{if}\ s_1\neq s_2,
      \end{cases}
    \end{align*}
    for some diagonal entries of diagonal blocks $\rho_{t}^{(s)}$, off-diagonal entries of diagonal blocks $\rho_{t,o}^{(s)}$, and entries of off-diagonal blocks $\rho_{s_1,s_2}^{(s)}$,
    do the following.
    \begin{enumerate}
    \item Set
      \begin{align}\label{Cholesky factor of CS matrix}
        L_{ss}=\begin{pmatrix}
          \tilde  l_{ss,1} & 0 & 0 & \cdots & 0\\
          l_{ss,1} & \tilde  l_{ss,2} &0 &   & \vdots\\
          l_{ss,1} & l_{ss,2} & \tilde l_{ss,3}& \ddots & \vdots\\
          \vdots & \vdots & \vdots & \ddots &0\\
          l_{ss,1} & l_{ss,2} & l_{ss,3} & \cdots & \tilde l_{ss,d_s}\\
        \end{pmatrix}\in\IR^{d_s\times d_s},
      \end{align}
      where
      \begin{align}\label{l-equations}
        \tilde l_{ss,j} = \sqrt{\rho_{1}^{(s)}-\sum_{k=1}^{j-1}l_{ss,k}^2},\quad
        \text{and}\quad l_{ss,j} = \frac{1}{\tilde l_{ss,j}}\biggl(\rho_{1,o}^{(s)} - \sum_{k=1}^{j-1}l_{ss,k}^2\biggr),\quad j=1,\dots,d_s.
      \end{align}
    \item If $s < S$, set, for $m=1,\dots,S-s$,
      \begin{align*}
        L_{s+m, s} = (c_{sm,1}\bm{1}_{d_{s+m}},\dots,c_{sm,d_s}\bm{1}_{d_{s+m}})\in\IR^{d_{s+m}\times d_s},
      \end{align*}
      where $(c_{sm,1},\dots,c_{sm,d_s})$ can be sequentially determined via
      \begin{align}\label{c-equations}
        c_{sm,j}\tilde l_{ss,j} + \sum_{k=1}^{j-1}c_{sm,k}l_{ss,k} = \rho_{m+1,1}^{(s)},\quad j=1,\dots,d_s.
      \end{align}
    \item If $s < S$, set ${\overline P}(s+1)$ to be of form \eqref{block correlation matrix at iteration s} with
      \begin{align*}
        \rho_{t}^{(s+1)} &= \rho_{t+1}^{(s)}  + \frac{ d_s (\rho_{t+1,1}^{(s)})^2 }{\rho_1^{(s)}+(d_s-1)\rho_{1o}^{(s)}},\\
        \rho_{t,o}^{(s+1)} &= \rho_{t+1,o}^{(s)}+\frac{ d_s (\rho_{t+1,1}^{(s)})^2 }{\rho_1^{(s)}+(d_s-1)\rho_{1o}^{(s)}}
      \end{align*}
      for $t \in \{1,\dots,S-s\}$ and
      \begin{align*}
        \rho_{s_i,s_j}^{(s+1)} = \rho_{s_i+1,s_j+1}^{(s)} + \frac{ d_s \rho_{s_i+1,1}^{(s)}\rho_{s_j+1,1}^{(s)} }{\rho_1^{(s)}+(d_s-1)\rho_{1o}^{(s)}}
      \end{align*}
      for $s_1,s_2 \in \{1,\dots,S-s\}$.
    \end{enumerate}
  \item Return the Cholesky factor $L$ of $P$ where
    \begin{align*}
      L =  \begin{pmatrix}
        L_{11} & O & \cdots & O \\
        L_{21} & L_{22} & \ddots & O \\
        \vdots & \vdots & \ddots & \vdots\\
        L_{S1} & L_{S2} & \cdots & L_{SS} \\
      \end{pmatrix}.
    \end{align*}
  \end{enumerate}
\end{algorithm}
\begin{proof}
Consider the first iteration $s=1$.
Let
\begin{align*}
  L_{11}=\sqrt{P_{11}},\quad L_{s1}=P_{s1} (L_{11}\T)^{-1},\quad s=2,\dots,d_{S}.
\end{align*}
Since $P_{11}=\Gamma_{d_1}(1,\rho_{11})$ is a compound symmetry matrix, solving the equation $L_{11} L_{11}\T=P_{11}$ yields that $L_{11}$ is of the form
\begin{align*}
  L_{11}=\begin{pmatrix}
    1 & 0 & 0 & \cdots & 0\\
    l_{11,1} & \tilde  l_{11,2} &0 &   & \vdots\\
    l_{11,1} & l_{11,2} & \tilde l_{11,3}& \ddots & \vdots\\
    \vdots & \vdots & \vdots & \ddots &0\\
    l_{11,1} & l_{11,2} & l_{11,3} & \cdots & \tilde l_{11,d_1}\\
    \end{pmatrix}\in\IR^{d_1\times d_1},
\end{align*}
where
\begin{align*}
\tilde l_{11,j} = \sqrt{1-\sum_{k=1}^{j-1}l_{11,k}^2},\quad
  \text{and}\quad l_{11,j} = \frac{1}{\tilde l_{11,j}}\biggl(\rho_{11} - \sum_{k=1}^{j-1}l_{11,k}^2\biggr),\quad j=1,\dots,d_1.
\end{align*}
Note that all off-diagonal components in the same column are equal.
This set of equations can be solved sequentially for $j=1,\dots,d_1$.
For $s=2,\dots,d_{S}$, since $P_{s1}=\rho_{s1}J_{d_sd_1}$,
$L_{s1}=\rho_{s1}J_{d_sd_1}(L_{11}\T)^{-1}$ can be written as
\begin{align*}
  L_{s1} = (c_{s1,1}\bm{1}_{d_s},\dots,c_{s1,d_1}\bm{1}_{d_s})\in\IR^{d_s\times d_1},\quad s=2,\dots,S,
\end{align*}
where $(c_{s1,1},\dots,c_{s1,d_1})$ can be sequentially determined via
\begin{align*}
  c_{s1,j}\tilde l_{11,j} + \sum_{k=1}^{j-1}c_{s1,k}l_{11,k} = \rho_{s1},\quad j=1,\dots,d_1.
\end{align*}
Let $P_{-(1:d_1)}$ be the submatrix of $P$ obtained by removing the first $d_1$ rows and
columns.  Let $L_{-1}$ be the Cholesky factor of
\begin{align*}
\overline P(1) = P_{-(1:d_1)} - (P_{21},\dots,P_{d_s1})\T P_{11}^{-1}(P_{21}\T,\dots,P_{d_S1}\T).
\end{align*}
Then $LL\T = P$ for the lower triangle matrix
\begin{align*}
  L=\begin{pmatrix}
      L_{11} & O & \cdots & O \\
      L_{21} &  & &  \\
      \vdots &  & L_{-1} &  \\
      L_{S1} &  &  &  \\
    \end{pmatrix}.
\end{align*}

We now show that $\overline P(1)$ is a %
block matrix with diagonal blocks equal to compound symmetric matrices and off-diagonal blocks equal to constant matrices.
Since
\begin{align*}
  &\phantom{{}={}}(P_{21},\dots,P_{S1})\T P_{11}^{-1}(P_{21}\T,\dots,P_{d_S1}\T) \\
  &=(\rho_{21}J_{d_2d_1},\dots,\rho_{S1}J_{d_Sd_1})\T P_{11}^{-1}(\rho_{21}J_{d_2d_1}\T,\dots,\rho_{S1}J_{d_Sd_1}\T),
\end{align*}
its $(i,j)$-block for $i,j \in \{1,\dots,S-1\}$ is given by
\begin{align*}
  \rho_{i+1,1}\rho_{j+1,1}J_{d_{i+1,}d_1}P^{-1}_{11}J_{d_{j+1}d_1}\T = \rho_{i+1,1}\rho_{j+1,1}\bm{1}_{d_{i+1}}\bm{1}_{d_1}\T P^{-1}_{11}\bm{1}_{d_1}\bm{1}_{d_{j+1}}\T.
\end{align*}
Since $P_{11} = \Gamma_{d_1}(1,\rho_{11})$, we have that
\begin{align*}
  P_{11}\bm{1}_{d_1} = (1+(d_1-1)\rho_{11})\bm{1}_{d_1}.
\end{align*}
Moreover,
\begin{align*}
  \bm{1}_{d_1}\T P_{11}^{-1}P_{11} \bm{1}_{d_1} = \bm{1}_{d_1}\T\bm{1}_{d_1} = d_1.
\end{align*}
Putting these equalities together, we obtain that
\begin{align*}
  \bm{1}_{d_1}\T P_{11}^{-1}\bm{1}_{d_1}=\frac{d_1}{1+(d_1-1)\rho_{11}}.
\end{align*}
Therefore, the $(i,j)$-block of the second term of $\overline P_{-1}$ is given by
\begin{align*}
  \rho_{i+1,1}\rho_{j+1,1}J_{d_{i+1}d_1}P^{-1}_{11}J_{d_{j+1}d_1}\T
  =\frac{ d_1 \rho_{i+1,1}\rho_{j+1,1} }{1+(d_1-1)\rho_{11}}J_{d_{i+1}d_{j+1}}.
\end{align*}
Consequently, $\overline P(1)$ is a block matrix with $(i,i)$th block given by
\begin{align*}
  &\phantom{{}={}}\Gamma_{d_{i+1}}(1,\rho_{i+1,i+1}) + \frac{ d_1 \rho_{i+1,1}^2 }{1+(d_1-1)\rho_{11}}J_{d_{i+1}} \\
  &= \Gamma_{d_{i+1}}\biggl(1  + \frac{ d_1 \rho_{i+1,1}^2 }{1+(d_1-1)\rho_{11}},\ \rho_{i+1,i+1}+\frac{ d_1 \rho_{i+1,1}^2 }{1+(d_1-1)\rho_{11}}\biggr),\quad i=1,\dots,S-1,
\end{align*}
and with $(i,j)$th block given by
\begin{align*}
  &\phantom{{}={}}\rho_{i+1,j+1}J_{d_{i+1},d_{j+1}} + \frac{ d_1 \rho_{i+1,1}\rho_{j+1,1} }{1+(d_1-1)\rho_{11}}J_{d_{i+1},d_{j+1}}\\
  &= \left(\rho_{i+1,j+1} + \frac{ d_1 \rho_{i+1,1}\rho_{j+1,1} }{1+(d_1-1)\rho_{11}}\right)J_{d_{i+1}d_{j+1}},\quad i,j \in \{1,\dots,S-1\},\ i \neq j.
\end{align*}
Since $\overline P(1)$ has the same structure as the initial matrix $P$, the same procedure can be applied to find a Cholesky factor $L_{-1}$ such that $L_{-1}L_{-1}\T = \overline P(1)$.
By iteratively applying this procedure, we obtain the Cholesky factor $L$ of $P$.
\end{proof}

Algorithm~\ref{algorithm: cholesky decomposition for block matrices} uses only $S(S+1)/2$ coefficients and the block sizes $\{d_1,\dots,d_S\}$ without the need to consider the full $d\times d$ matrix $P$,
which can lead to significant computational savings
especially when $d$ is large and $S$ is small.
The following example covers the individual steps of Algorothm~\ref{algorithm: cholesky decomposition for block matrices} with concrete numbers.

\begin{example}[Case of $S=3$, $(d_1,d_2,d_3)=(4,3,2)$]
  Consider the block homogeneous matrix \eqref{eq:hier:mat}.
  As discussed in Example~\ref{example: Positive definiteness of a hierarchical
    matrix}, the matrix $P$ in \eqref{eq:hier:mat} is positive definite, and
  thus, has a Cholesky factor $L$.  By applying Algorithm~\ref{algorithm:
    cholesky decomposition for block matrices}, the Cholesky factor $L$ of $P$
  is obtained as
  \begin{align*}
    P=\left(\ \ \begin{matrix}
      \cc{0.5} 1   & 0    & 0    & 0    & 0    & 0    & 0    & 0    & 0\\
      \cc{0.8} 0.4 & \cc{0.66} 0.92 & 0    & 0    & 0    & 0    & 0    & 0    & 0\\
      \cc{0.8} 0.4 & \cc{0.85} 0.26 & \cc{0.74} 0.88 & 0    & 0    & 0    & 0    & 0    & 0\\
      \cc{0.8} 0.4 & \cc{0.85} 0.26 & \cc{0.9} 0.2  & \cc{0.78} 0.86 & 0    & 0    & 0    & 0    & 0\\
      \cc{0.9} 0.1 & \cc{0.94} 0.07 & \cc{0.97} 0.05 & \cc{1} 0.04 & \cc{0.52} 0.99 & 0    & 0    & 0    & 0\\
      \cc{0.9} 0.1 & \cc{0.94} 0.07 & \cc{0.97} 0.05 & \cc{1} 0.04 & \cc{0.85} 0.28 & \cc{0.60} 0.95 & 0    & 0    & 0\\
      \cc{0.9} 0.1 & \cc{0.94} 0.07 & \cc{0.97} 0.05 & \cc{1} 0.04 & \cc{0.85} 0.28 & \cc{0.9} 0.21 & \cc{0.64} 0.93 & 0    & 0\\
      \cc{0.9} 0.1 & \cc{0.94} 0.07 & \cc{0.97} 0.05 & \cc{1} 0.04 & \cc{0.92} 0.13 & \cc{0.96} 0.1 & \cc{1} 0.08 & \cc{0.56} 0.97 & 0\\
      \cc{0.9} 0.1 & \cc{0.94} 0.07 & \cc{0.97} 0.05 & \cc{1} 0.04 & \cc{0.92} 0.13 & \cc{0.96} 0.1 & \cc{1} 0.08 & \cc{0.92} 0.15 & \cc{0.58} 0.96\\
    \end{matrix}\ \ \right).
  \end{align*}
  In the first iteration $s=1$ of Algorithm~\ref{algorithm: cholesky
    decomposition for block matrices} with $\overline P(1)=P$, Cholesky factor
  in the first
  $d_1=4$ columns is computed.  By solving \eqref{l-equations},
  $P_{11}=\Gamma_{d_1}(1,\rho_{11})$ is decomposed into $L_{11}$ of form
  \eqref{Cholesky factor of CS matrix}, which is determined by
  $(\tilde l_{11,1},\tilde l_{11,2},\tilde l_{11,3},\tilde
  l_{11,4},l_{11,1},l_{11,2},l_{11,3})=(1.00,0.92,0.88,0.88,0.40,0.26,0.20)$.
  By solving \eqref{c-equations}, $L_{21}$ and $L_{31}$ are determined via
  $(c_{11,1},\dots,c_{11,d_{1}})$ and $(c_{12,1},\dots,c_{12,d_{1}})$ by
  $(c_{11,1},\dots,c_{11,4})=(c_{12,1},\dots,c_{12,4})=(0.1,0.07,0.05,0.04)$.
  For iteration $s=2$, the submatrix $\overline P(2)$ is computed following Step 5)
  via
  $(\rho_{1}^{(2)},\rho_{1,o}^{(2)},\rho_{2}^{(2)},\rho_{2,o}^{(2)},\rho_{12}^{(2)})=(0.98,0.28,0.98,0.18,0.13)$.
  By solving \eqref{l-equations} and \eqref{c-equations}, $L_{22}$ and $L_{32}$
  are specified via
  $(\tilde l_{22,1},\tilde l_{22,2},\tilde
  l_{22,3},l_{22,1},l_{22,1})=(0.99,0.95,0.93,0.28,0.21)$, and
  $(c_{21,1},c_{21,2})=(0.13,0.10)$.  Finally, the submatrix $\overline P(3)$ is
  given by $\overline P(3)=\Gamma_{2}(0.95,0.15)$. The Cholesky factor $L_{33}$ is
  then specified via
  $({\tilde l}_{33,1},{\tilde l}_{33,2},l_{33,1})=c(0.97,0.96,0.15)$ by solving
  the equations in \eqref{l-equations}.
\end{example}

\subsection{Attainability for block matrices}
In this section, we study compatibility and attainability of measures of
concordance for a block homogeneous matrices of form \eqref{block correlation structure}.
We expect that checking compatibility and attainability of a given $d\times d$
block matrix can be reduced to check those of some $S\times S$ matrix for a block
size $S$, which can be much smaller than $d$.

For van der Waerden's coefficient, we have already seen that
Theorem~\ref{theorem: iff condition for positive semi-definiteness of block
  correlation matrix} is available for checking compatibility and that
Proposition~\ref{proposition: Cholesky factor of block matrices} is beneficial
to attain a given $\zeta$-compatible matrix.  For Spearman's rho block matrices, we have the following
result.

\begin{proposition}[$\rho_\text{S}$-compatible subclass of block matrices]\label{proposition: Compatibility criterion for block matrix of spearmans rho}
  Let $P$ be a $d_1+\cdots+d_S$ block homogeneous correlation matrix of form \eqref{block correlation structure}.
  Let $M =(m_{s_ks_l})$ be a $S\times S $ matrix with $m_{ss}=1$, $s=1,\dots,S$, and
  \begin{align*}
    m_{s_ks_l} = \frac{d_{s_k} d_{s_l} \rho_{s_k,s_l}}{(1+(d_{s_k} - 1 )\rho_{s_ks_k})(1+(d_{s_l} - 1 )\rho_{s_ls_l})},\quad s_k,s_l \in \{1,\dots,S\},\ s_k \neq s_l.
  \end{align*}
  If $M\in \mathcal S_S$, then $P$ is $\rho_{\text{S}}$-compatible.
  Moreover, if $M$ is $\rho_{\text{S}}$-attainable, so is $P$.
\end{proposition}
\begin{proof}
  Let $\lambda_s=\tilde \rho_{ss}=\frac{1+(d_s - 1)\rho_{ss}}{d_s}$.
  Then positive definiteness of $P$ requires $-1/(d_s - 1)< \rho_{ss}< 1$ and thus it holds that $\lambda_s \in (0,1)$.
  Notice that
  \begin{align*}
    \lambda_{s} + (1-\lambda_s)\left( -\frac{1}{d_s - 1}\right) = \rho_{ss}.
  \end{align*}
  If $M \in \mathcal S_S$, there exists an $S$-dimensional
  random vector $\bm{U}=(U_1,\dots,U_S)$ with standard uniform margins such that
  $\rho(\bm{U})= M$.  For $s\in \{1,\dots,S\}$, there exists a $d_s$-dimensional
  random vector $\bV_{s}$ with $\U(0,1)$ margins such that its correlation matrix is
  $\Gamma(1,-1/(d_s-1))$ for $s\in \{1,\dots,S\}$; see
  \cite{murdoch2001edge} for a construction.  Let $\bV_1,\dots,\bV_S$ be such random vectors
  independent of each other, and also independent of
  $\bm{U}$.  For $s\in \{1,\dots,S\}$, let $B_s \sim \Bern(\lambda_{s})$ such
  that $B_1,\dots,B_S$ are independent of each other, and independent of
  $\bm{U}$ and $\bV_1,\dots,\bV_S$.  For $s=1,\dots,S$, define a
  $d_s$-dimensional random vector
  \begin{align}\label{W vector}
    \bW_s = B_{s}U_{s}\bm{1}_{d_s}+ (1-B_s)\bV_s.
  \end{align}
  One can easily check that $\bW_s$ has $\U(0,1)$ marginals.
  Moreover, for $s=1,\dots,S$,
  \begin{align*}
    \rho(\bW_s) = \lambda_{s}J_{d_{s}} + (1-\lambda_{s})\Gamma_{d_{s}}(1,-1/(d_s - 1))
    = \Gamma_{d_s}(1,\rho_{ss})=P_{ss},
  \end{align*}
  and for $s_1\neq s_2$, $i=1,\dots,d_{s_1}$, $j=1,\dots,d_{s_2}$,
  \begin{align*}
    \rho(W_{s_1i},W_{s_2j}) = \lambda_{s_1}\lambda_{s_2}\rho(U_{s_1},U_{s_2})
    =  \lambda_{s_1}\lambda_{s_2}m_{s_1s_2}
    = \rho_{s_1s_2}.
  \end{align*}
  Therefore, $(\bW_1\T,\dots,\bW_S\T)$ is a $(d_1 + \cdots + d_S)$-dimensional
  random vector with correlation matrix $P$.  Since its marginal distributions
  are all $\U(0,1)$, $P$ is $\rho_{\text{S}}$-compatible by Proposition~\ref{prop:characterize:S:B:W}~\ref{prop:characterize:S:B:W:2}.  If $M$ is $\rho_{\text{S}}$-attainable by constructing $\bU$
  above, then $P$ is $\rho_{\text{S}}$-attainable via construction \eqref{W vector}.
\end{proof}

If $S\leq 9$, checking $M\in \mathcal S_S$ can be reduced to checking its positive
semi-definiteness by
Proposition~\ref{prop:characterize:S:B:W}~\ref{prop:characterize:S:B:W:1} and
\ref{prop:characterize:S:B:W:2}.  If $S\geq 10$, a sufficient condition is
available related to $\Bern(1/2)$-compatibility by Proposition
\ref{prop:suff:cond:k:compatibility}.  On attainability of $P$, $M$ is
$\rho_{\text{S}}$-attainable only for the sector size $S=3$; see the discussion
of $\rho_{\text{S}}$-attainability in Section~\ref{subsec: Attainability}.

\begin{example}[Case with $d=9$ and $S=3$]
Let $P$ be the block homogeneous correlation matrix defined in \eqref{eq:hier:mat}.
Since $d\leq 9$, its compatibility can be verified by checking that $P$ is positive semi-definite.
In fact, the corresponding matrix $M$ in Proposition~\ref{proposition: Compatibility criterion for block matrix of spearmans rho} of $P$ is
\begin{align*}
 M = \begin{pmatrix}
    1 & 0.341 & 0.303 \\
    0.341 & 1 & 0.469 \\
    0.303 & 0.469 & 1 \\
  \end{pmatrix}
\end{align*}
and one can also check that $M$ is positive definite by a simple calculation.
Therefore, $P$ is $\rho_{\text{S}}$-compatible by Proposition~\ref{proposition:
  Compatibility criterion for block matrix of spearmans rho}.  Since $M$ is
$3$-dimensional, $P$ is $\rho_{\text{S}}$-attainable; see the discussion in Subsection~\ref{subsec: Attainability}.
Therefore, even though
$P$ is 9 $(> 3)$-dimensional, it is $\rho_{\text{S}}$-attainable by construction~\eqref{W vector}.
\end{example}

When a given block homogeneous matrix $P$ is a hierarchical matrix, then the
following sufficient condition is available for compatibility and attainability
of \emph{any} measure of concordance.

\begin{proposition}[Compatible and attainable hierarchical matrices]\label{prop:compatibility:hierarchical:mat}
  For a general measure of concordance $\kappa$, a $d\times d$ hierarchical
  matrix $P$ is $\kappa$-compatible and $\kappa$-attainable (by a
    nested or hierarchical Archimedean copula (HAC)) if, for the corresponding hierarchical
  tree, $0\leq\rho_{v}\leq\rho_{v'}$ holds for every pair of nodes $(v,v')$ such
  that $v$ is a parent of $v'$.
\end{proposition}

\begin{proof}
  Let $\psi_{\theta}:[0,\infty]\rightarrow [0,1]$ be a one-parameter Archimedean generator
  with $\theta \in \Theta=(\theta_{\text{min}},\theta_{\text{max}})$,
  $\theta_{\text{min}}\le\theta_{\text{max}}\le\infty$ and let
  $C_{\theta}(u_1,u_2)=\psi_{\theta}(\psi_{\theta}^{-1}(u_1)+\psi_{\theta}^{-1}(u_2))$, $u_1,u_2\in[0,1]$,
  be the corresponding Archimedean copula family. Suppose
  $\{\psi_{\theta};\theta\in \Theta\}$ satisfies the following conditions:
  \begin{enumerate}
  \item[(1)](Complete monotonicity) $(-1)^{k}\frac{\rd ^k}{\rd t^{k}}\psi_{\theta}(t)\geq 0$ for any $\theta \in \Theta$ and $k=0,1,\dots$;
  \item[(2)](Limiting copulas) $C_{\theta_{\text{min}}}=\lim_{\theta\downarrow\theta_{\text{min}}}C_{\theta}$ is the independence copula and $C_{\theta_{\text{max}}}=\lim_{\theta\uparrow\theta_{\text{max}}}C_{\theta}$ is the comonotone copula;
  \item[(3)](Positive ordering) if $\theta,\theta' \in \Theta$ such that $\theta\leq \theta'$ then $C_{\theta}\preceq C_{\theta'}$; and
  \item[(4)](Sufficient nesting condition) $\psi_{\theta}^{-1}\circ\psi_{\theta'}$ is completely monotone for $\theta,\theta' \in \Theta $ if and only if $\theta\leq \theta'$.
  \end{enumerate}
  Examples of Archimedean copulas satisfying Conditions (1)--(4) are the Clayton
  and Gumbel copula families with generators given by Laplace
    transforms of certain gamma and positive stable distributions,
    respectively; see \cite[Examples~4.12 and 4.14]{nelsen2006} and
  \cite[Tables~2.1 and 2.3]{hofert2010c}. Note that Condition (1) guarantees
  that the $d$-dimensional Archimedean copula
  $C_{\theta}(u_1,\dots,u_d)=\psi_{\theta}(\sum_{j=1}^{d}\psi_{\theta}^{-1}(u_j))$
  is also a $d$-copula for any $d\geq 2$; see \cite{kimberling1974}.  Together
  with the continuity and coherence axioms of a measure of concordance,
  Condition (2) and (3) imply that the map
  $\kappa(\theta):\theta \mapsto \kappa(C_{\theta})$ is increasing and
  continuous from $\Theta$ to $[0,1]$.  Therefore, for every pair of nodes
  $(v,v')$, there exist $\theta_{v},\theta_{v'}\in \Theta$ such that
  $\theta_{v} \leq \theta_{v'}$ and
  $\kappa(\theta_v)=\rho_v\leq \rho_{v'}= \kappa(\theta_{v'})$.  For the
  hierarchical tree $T_P$ of a given hierarchical matrix $P$ with the
  corresponding collection of generators $\{\psi_{\theta_v};v \in \mathcal V\}$,
  Condition (4) thus ensures that there exists a corresponding HAC; see
  \cite{mcneil2008} and \cite[pp.~87]{joe1997} for the sufficient nesting
  condition and \cite{hofert2012b} and \cite{goreckihofertholena2017b} for the
  construction of HACs. By construction, the matrix of pairwise measure of
  concordance $\kappa$ is equal to $P$ for this HAC. Thus, $P$ is both
  $\kappa$-compatible and $\kappa$-attainable.
\end{proof}

When a hierarchical matrix $P$ satisfies the sufficient condition in
Proposition~\ref{prop:compatibility:hierarchical:mat}, we call $P$ a
\emph{proper hierarchical matrix}.  Note that componentwise non-negativity of $P$ is
necessary since complete monotonicity (1) of $\psi_\theta$ implies
that $\Pi\preceq C_{\theta}$; see \cite[Remark~2.3.2]{hofert2010c}.
For sampling from a HAC, see \cite{mcneil2008}, \cite{hofert2011a} or \cite{hofert2012b}.

\begin{remark}[Positive definiteness of hierarchical matrices]
  In Proposition~\ref{prop:compatibility:hierarchical:mat}, positive
  definiteness of $P$ was not a necessary assumption. In fact, positive
  definiteness is impled by the condition $0\leq \rho_v\leq \rho_{v'}$ for any
  $v$ and $v'$ such that $v'$ is a parent of $v$ since
  Proposition~\ref{prop:compatibility:hierarchical:mat} holds for any
  $G$-transformed rank correlation coefficient and $\kappa_G$-compatible
  matrices are necessarily positive definite.
\end{remark}

\begin{example}[Attainability of hierarchical matrix~\eqref{eq:hier:mat} for general $\kappa$]
  By Proposition~\ref{prop:compatibility:hierarchical:mat}, the hierarchical
  matrix $P$ in \eqref{eq:hier:mat} is $\kappa$-compatible and $\kappa$-attainable for any
  measure of concordance $\kappa$ since $P$ is proper as can be easily checked
  from Figure~\ref{fig:tree:hier:mat}. As an example of a model attaining $P$,
  let $\psi_{\theta}$ be the generator of Gumbel copula and let $C_P$ be the
  corresponding HAC given, for each $\bm{u}\in[0,1]^9$, by
  \begin{align*}
    C_P (u_1,\dots,u_9)=C_{v_{01}}\bigl(C_{v_{11}}(u_1,u_2,u_3,u_4),C_{v_{12}}(C_{v_{21}}(u_5,u_6,u_7),C_{v_{22}}(u_8,u_9))\bigr),
  \end{align*}
  where the Gumbel copula $C_v$ has parameter $\theta_v$ such that
  $\kappa(C_{v})=\rho_v$ is attained for every node $v$.  For example, if
  $\kappa$ is Blomqvist's beta $\beta$, one has
  $\beta(\theta_v)=\beta(C_v)=4C_v(1/2,1/2)-1=2^{2-2^{1/\theta_v}}-1$,
  $\theta_v \in [1,\infty)$, which is continuous and increasing
  from $0$ to $\lim_{\theta_v \rightarrow \infty}\beta(\theta_v)=1$. Therefore, for
  each $\rho_v=\beta_{v}$, $v \in \mathcal V$, the parameter $\theta_v$ is given by
  $\theta_v = 1/\bigl(\log_2(2-\log_2(1+\beta_v))\bigr)$.

  As an another example, when $\kappa$ is Kendall's tau $\tau$, it is known that
  $\tau(\theta_v)=\tau(C_{\theta_v})=(\theta_v-1)/\theta_v$ for
  $\theta_v\in\Theta=[1,\infty)$ and so $\theta_v=1/(1-\tau_v)$ where $\tau_v$
  is the corresponding entry in $P$ in~\eqref{eq:hier:mat} or
  Figure~\ref{fig:tree:hier:mat}.  Thus, for example, $\tau_{v_{01}}=0.1$
  implies that $\theta_{v_{01}}=10/9$. The same construction applies to $\kappa$
  being Spearman's rho or van der Waerden's coefficient and the $C_{v}$ being
  Clayton copulas, for example. Note that it may sometimes be necessary to find
  $\theta_v$ such that $\kappa(\theta_v)=\kappa_v$ for a given $\kappa_v$ numerically.
\end{example}

\section{Conclusion and discussion}
We introduced a new class of measures of concordance called transformed rank
correlation coefficients, whose members depend on functions $G_1$ and
$G_2$. Spearman's rho, Blomqvist's beta and van der Waerden's coefficient are
obtained as special cases. We provided necessary and sufficient conditions on
$G_1$ and $G_2$ when transformed rank correlation coefficients are measures of
concordance; see Theorem~\ref{thm:G1:G2}.

For matrices of (pairwise) transformed rank correlation coefficients, a
sufficient condition for compatibility and attainability was derived in terms of
$\Bern(1/2)$-compatibility; see Proposition~\ref{prop:suff:cond:k:compatibility}
and Corollary~\ref{corollary: kappa attainability} for compatibility and
attainability, respectively.  We also presented characterizations of the sets of
compatible Spearman's rho, Blomqvist's beta and van der Waerden's matrices; see
Proposition~\ref{prop:characterize:S:B:W}.  This result revealed that, among
these measures of concordance, van der Waerden's coefficient may be the most convenient one
in terms of checking compatibility and attainability.

We then studied compatible and attainable block matrices for which fast methods
of checking positive semi-definiteness and of calculating Cholesky factors were
derived; see Theorem~\ref{theorem: iff condition for positive semi-definiteness
  of block correlation matrix} and Algorithm~\ref{algorithm: cholesky
  decomposition for block matrices}, respectively.  For certain subclasses of
block matrices, the problem of checking compatibility and attainability can be
reduced to lower dimensions; see Proposition~\ref{proposition: Compatibility
  criterion for block matrix of spearmans rho} and
Proposition~\ref{prop:compatibility:hierarchical:mat}.

While hierarchical Kendall's tau matrices with non-negative entries are
attainable, Kendall's tau is not a transformed rank correlation
coefficient. This gives rise to the open question of compatibility and
attainability of Kendall's tau matrices.
Finding a wider class of measures of
concordance including Kendall's tau and other concordance measures such as Gini's
coefficient could help in providing an answer to this
question.
An another angle to take for future research concerns a comparison among different
transformed rank correlation coefficients to obtain a clear answer on which
measure is the best to be used from a statistical point of view. In terms of
block matrices, dimension reduction for (computationally) checking compatibility
of general transformed rank correlation coefficients is also an interesting problem for
future research.

\section*{Acknowledgements}
We thank Ruodu Wang (University of Waterloo) for fruitful discussions on an
early version of the manuscript and Sebastian Fuchs (Technische Universit\"at
Dortmund) for his comments on Theorem~\ref{thm:G1:G2} and
Remark~\ref{rem:D4inv}. The first author would also like to thank NSERC for
financial support through Discovery Grant RGPIN-5010-2015.

\printbibliography[heading=bibintoc]

\bigskip
\noindent{\scshape Marius Hofert}\\
{\itshape Department of Statistics and Actuarial Science\\
  University of Waterloo\\
  Waterloo, ON, Canada\\
Email:}\ \href{mailto:marius.hofert@uwaterloo.ca}{\nolinkurl{marius.hofert@uwaterloo.ca}}

\bigskip
\noindent{\scshape Takaaki Koike}\\
{\itshape Department of Statistics and Actuarial Science\\
  University of Waterloo\\
  Waterloo, ON, Canada\\
Email:}\ \href{mailto:tkoike@uwaterloo.ca}{\nolinkurl{tkoike@uwaterloo.ca}}

\appendix

\section{Measures of concordance which cannot be represented as $\kappa_G$}\label{app:B}
As discussed in Remark~\ref{remark:degree}, any measure of concordance which has
degree more than one is not included in the set of $G$-transformed rank
correlations. In this section, we briefly provide examples of such measures of
concordance which are not $G$-transformed rank correlations.

To this end, consider Kendall's tau $\tau$ and Gini's coefficient $\gamma$ defined by
\begin{align*}
  \tau(X_1,X_2)&=4\int_{[0,1]^2}C(u,v)\,\rd C(u,v)-1,\\
  \gamma(X_1,X_2)&=4\int_{[0,1]^2}(M(u,v)+W(u,v))\,\rd C(u,v)-2,
\end{align*}
respectively.
The $G$-transformed rank correlation coefficient can be written as
\begin{align*}
  k_{G}(C) =\frac{1}{\sigma^2} \int_{[0,1]^2} G^\i(u)G^\i(v)\,\rd C(u,v) -\left(\frac{\mu}{\sigma}\right)^2,
\end{align*}
where $\mu\in\IR$ and $\sigma>0$ are the mean and standard deviation of $G$,
respectively. This expression implies that the integrand with respect to the
underlying copula $C$ must be of the product form $G^\i(u)G^\i(v)$.  Since the
integrands of $\tau$ and $\gamma$ cannot be decomposed into such a
product form in general, these measures of concordance are not $G$-transformed
rank correlation coefficients.

\section{Open problem for compatibility of Kendall's tau matrices}\label{app:C}
It is challenging to characterize the sets of compatible and attainable matrices
for Kendall's tau and Gini's coefficient since they cannot be written as
$G$-transformed rank correlation coefficients. The proof of
Proposition~\ref{prop:compatibility:hierarchical:mat} also applies to $\tau$, so
proper hierarchical matrices are $\tau$-compatible and $\tau$-attainable.  In
this section we present some partial results on Kendall's tau compatibility for
general matrices.

Denote by $\mathcal T_d$ the set of all Kendall's tau matrices attained
by continuous $d$-random vectors. The following result stems
from the definition of Kendall's tau.
\begin{proposition}[A necessary condition for $\tau$-compatibility]\label{partial characterization of Kendalls tau matrix}
  $\mathcal T_d \subseteq \mathcal P_{d}^{\text{B}}(1/2)$, that is, any
  Kendall's tau matrix is a correlation matrix of some $d$-random vector with $\Bern(1/2)$
  margins.
\end{proposition}
\begin{proof}
  Fix $(\tau_{ij})\in \mathcal T_d$.  Then there exists a $d$-random vector
  $\bm{X}=(X_1,\dots,X_d)$ with continuous margins $F_1,\dots,F_d$ such that
  $\tau(X_i,X_j)=\tau_{ij}$ for all $i,j\in \{1,\dots,d\}$.  Let
  $\bm{U}=(U_1,\dots,U_d)=(F_1(X_1),\dots,F_d(X_d))$.  If $\bm{X}$ has copula
  $C$, then $\bm{U}\sim C$ by continuity of $F_1,\dots,F_d$.  Let
  $\tilde{\bm{U}}\sim C$ be an independent copy of $\bm{U}$ and define
  $\bm B = (B_1,\dots,B_d)$ with $B_j = \I_{\{U_j\le\tilde U_j\}}$,
  $j=1,\dots,d$. Since $U_j$ and $\tilde U_j$ are independent and
  identically distributed with $\P(\I_{\{U_j\le\tilde U_j\}}=1)=1/2$, we have
  $B_j \sim \Bern(1/2)$ for $j=1,\dots,d$.  Consequently, for
  $i,j\in \{1,\dots,d\}$,
  \begin{align*}
    \rho(B_i,B_j)&=4\E(\I_{\{U_i\le\tilde U_i\}}\I_{\{U_j\le\tilde U_j\}})-1
    =4 \int_{[0,1]^2}\E(\I_{\{U_i\le u_i\}}\I_{\{U_j\le u_j\}})\,\rd C(u_i,u_j)-1\\
    &=4\int_{[0,1]^2}C(u_i,u_j)\,\rd C(u_i,u_j) - 1 = \tau_{ij},
  \end{align*}
  where the second equation follows by conditioning on $\tilde \bU \sim C$
  independent of $\bU \sim C$.  Since $(\tau_{ij})$ is attained as a
  correlation matrix of a symmetric Bernoulli random vector $\bm B$, we conclude
  that $(\tau_{ij})\in \mathcal P_d^{\text{B}}(1/2)$.
\end{proof}

Proposition \ref{partial characterization of Kendalls tau matrix} provides a
necessary condition for a given matrix to be $\tau$-compatible.
Thus, a given matrix $P$ is $\tau$-incompatible if $P$ does not belong to $\mathcal P_d^{\text{B}}(1/2)$.
Together with Corollary~\ref{corollary:upper:lower:bounds}, one obtains that
$\mathcal T_d\subseteq \mathcal K_G$ for any concordance-inducing function $G$, that is, the set of $\tau$-compatible matrices is smaller than $ \mathcal K_G$ for any choice of $G$.

Whether
$\mathcal T_d=\mathcal P_{d}^{\text{B}}(1/2)$ or not is an open problem. When $d=3$,
\cite{joe1996families} showed that
$\mathcal T_3 = \mathcal P_3^{\text{B}}(1/2)$.  However, unfortunately his approach
does not extend to $d\geq 4$.

\end{document}

%
%
%
%
